\newcount\switch
\newcount\correct
\newcount\revision
\ifnum\switch=1

\RequirePackage{amsmath}
\documentclass[smallextended]{svjour3}

\usepackage[utf8]{inputenc}
\usepackage{amsfonts,amssymb,array, graphicx, epsfig, fancyhdr, stmaryrd,algpseudocode}
\usepackage{mathrsfs} 
 \usepackage{pgfplots}
\usepackage{tikz}
 \pgfplotsset{compat=1.15}
\usepackage{xcolor}
\usepackage{comment}
\usepackage[caption=false]{subfig}
\usepackage[ruled, vlined]{algorithm2e}

\define{.7,.7,1}
\definecolor{backcolor2}{rgb}{1,.7,0.7}
\usetikzlibrary{positioning}
\usepackage{enumitem}
\usepackage{hyperref}
\usepackage[round,sort&compress,colon]{natbib}
\usepackage[textsize=footnotesize]{todonotes}

\usetikzlibrary{shapes,backgrounds,calc}

\usetikzlibrary{decorations.pathreplacing}
\usetikzlibrary{arrows.meta}

\makeatletter
\tikzset{circle split part fill/.style  args={#1,#2}{%
 alias=tmp@name, 
  postaction={%
    insert path={
     \pgfextra{%
     \pgfpointdiff{\pgfpointanchor{\pgf@node@name}{center}}%
                  {\pgfpointanchor{\pgf@node@name}{east}}%
     \pgfmathsetmacro\insiderad{\pgf@x}
      \fill[#1] (\pgf@node@name.base) ([xshift=-\Aufgaben werden angezeigtpgflinewidth]\pgf@node@name.east) arc
                          (0:180:\insiderad-\pgflinewidth)--cycle;
      \fill[#2] (\pgf@node@name.base) ([xshift=\pgflinewidth]\pgf@node@name.west)  arc
                          (180:360:\insiderad-\pgflinewidth)--cycle;            
         }}}}}  
 \makeatother  

\pgfplotsset{select coords between index/.style 2 args={
    x filter/.code={
        \ifnum\coordindex<#1\fi
        \ifnum\coordindex>#2\fi
    }
}}
\tikzset{
    position/.style args={#1:#2 from #3}{
        at=(#3), anchor=#1+180, shift=(#1:#2)
    }
}

\newcommand{\id}{\mbox{Id}}

\newcommand{\code}[1]{\texttt{#1}}
\usepackage{amsmath}

\DeclareMathOperator{\spa}{span}
\DeclareMathOperator*{\argmin}{arg\,min}

\fi

\ifnum\switch=2
\documentclass{article}
\usepackage{arxiv}
\usepackage[square,sort&compress,colon,numbers]{natbib}
\usepackage[utf8]{inputenc}
\usepackage{amsthm,amsmath,amsfonts,amssymb,array, graphicx, epsfig, fancyhdr, stmaryrd,algpseudocode}
\usepackage{mathrsfs} 
 \usepackage{pgfplots}
\usepackage[textsize=footnotesize]{todonotes}
\usepackage{tikz}
\usepackage{xcolor}
\usepackage{comment}
\usepackage[caption=false]{subfig}
\definecolor{backcolor}{rgb}{.7,.7,1}
\definecolor{backcolor2}{rgb}{1,.7,0.7}
\usetikzlibrary{positioning}
\usepackage{enumitem}
\usepackage{hyperref}

\usetikzlibrary{shapes,backgrounds,calc}

\usetikzlibrary{decorations.pathreplacing}
\usetikzlibrary{arrows.meta}

\DeclareMathOperator{\spa}{span}

\newtheoremstyle{plainNoItalics}{}{}{\normalfont}{}{\bfseries}{.}{ }{}

\theoremstyle{plain}
\newtheorem{theorem}{Theorem}[section]

\newtheorem{lemma}[theorem]{Lemma}

\newtheorem*{theorem*}{Theorem}

\newtheorem*{lemma*}{Lemma}
\newtheorem*{corollary*}{Corollary}

\newtheorem*{observation*}{Observation}
\newtheorem*{example*}{Example}
\newtheorem*{assumption*}{Assumption}

\makeatletter
\tikzset{circle split part fill/.style  args={#1,#2}{%
 alias=tmp@name, 
  postaction={%
    insert path={
     \pgfextra{%
     \pgfpointdiff{\pgfpointanchor{\pgf@node@name}{center}}%
                  {\pgfpointanchor{\pgf@node@name}{east}}%
     \pgfmathsetmacro\insiderad{\pgf@x}
      \fill[#1] (\pgf@node@name.base) ([xshift=-\pgflinewidth]\pgf@node@name.east) arc
                          (0:180:\insiderad-\pgflinewidth)--cycle;
      \fill[#2] (\pgf@node@name.base) ([xshift=\pgflinewidth]\pgf@node@name.west)  arc
                           (180:360:\insiderad-\pgflinewidth)--cycle;            
         }}}}}  
 \makeatother  

\pgfplotsset{select coords between index/.style 2 args={
    x filter/.code={
        \ifnum\coordindex<#1\fi
        \ifnum\coordindex>#2\fi
    }
}}
\tikzset{
    position/.style args={#1:#2 from #3}{
        at=(#3), anchor=#1+180, shift=(#1:#2)
    }
}

\theoremstyle{definition}
\newtheorem{definition}{Definition}
\newtheorem{remark}[definition]{Remark}

\usepackage[ruled, vlined]{algorithm2e}
\theoremstyle{plain}
\newtheorem{proposition}[definition]{Proposition}

\newcommand{\id}{\mbox{Id}}

\newcommand{\code}[1]{\texttt{#1}}
\usepackage{amsmath}

\DeclareMathOperator*{\argmin}{arg\,min}

\fi

\title{Approximating the Stationary Bellman Equation by Hierarchical Tensor Products}
\author{Mathias Oster \and 
  Leon Sallandt \and Reinhold Schneider}
\date{Received: date / Accepted: date}
\ifnum\switch=2
\title{Approximating the Stationary Bellman Equation by Hierarchical Tensor Products}
\author{Mathias Oster\thanks{TU BERLIN} \and 
  Leon Sallandt \and Reinhold Schneider}
  
\author{
  Mathias Oster \\
            Technische Universit\"at Berlin\\
            Strasse des 17. Juni 135 \\
            10623 Berlin, Germany \\
  \texttt{oster@math.tu-berlin.de} \\
   \And
  Leon Sallandt\\
            Technische Universit\"at Berlin\\
            Strasse des 17. Juni 135 \\
            10623 Berlin, Germany \\
  \texttt{sallandt@math.tu-berlin.de} \\
  \And
  Reinhold Schneider\\
            Technische Universit\"at Berlin\\
            Strasse des 17. Juni 135 \\
            10623 Berlin, Germany \\
  \texttt{schneidr@math.tu-berlin.de} \\
}
\date{Received: date / Accepted: date}

\fi

\ifnum\switch=1
\title{Approximating the Stationary Bellman Equation by Hierarchical Tensor Products}
\author{Mathias Oster \and 
  Leon Sallandt \and Reinhold Schneider}
\date{Received: date / Accepted: date}
\titlerunning{Approximating the Stationary Bellman Equation}

\institute{
    Mathias Oster \at
            Technische Universit\"at Berlin, Strasse des 17. Juni 135, 10623 Berlin, Germany \\
    \email{\textbf{oster@math.tu-berlin.de}}\\
    ORCiD: 000-0001-6603-7883
              \and
          Leon Sallandt \at
            Technische Universit\"at Berlin, Strasse des 17. Juni 135, 10623 Berlin, Germany \\
          \email{sallandt@math.tu-berlin.de} \\
          ORCiD: 0000-0002-7102-7767
            \and
            Reinhold Schneider \at
            Technische Universit\"at Berlin, Strasse des 17. Juni 135, 10623 Berlin, Germany \\
            \email{schneidr@math.tu-berlin.de}
}

\fi
\begin{document}
\maketitle

\ifnum\switch=2

    \bibliographystyle{plain}
\fi

\ifnum\switch=1
   \bibliographystyle{spbasic}
\fi

\begin{abstract}
We treat infinite horizon optimal control problems by solving  the associated stationary Bellman equation numerically to compute the value function and an optimal feedback law. The dynamical systems under consideration are spatial discretizations of non linear parabolic partial differential equations (PDE), which means that the Bellman equation suffers from the curse of dimensionality. Its non linearity is handled by the Policy Iteration algorithm, where the problem is reduced to a sequence of linear equations, which remain the computational bottleneck due to their high dimensions.
We reformulate the linearized Bellman equations via the Koopman operator into an operator equation, that is solved using a minimal residual method.
Using the Koopman operator we identify a preconditioner for operator equation, which deems essential in our numerical tests.
To overcome computational infeasability we use low rank hierarchical tensor product approximation/tree-based tensor formats, in particular  tensor trains (TT tensors) and 
multi-polynomials, together with high-dimensional quadrature, e.g. Monte-Carlo.

By controlling a destabilized version of viscous Burgers and a diffusion equation with unstable reaction term numerical evidence is given.
\end{abstract}
\section{Introduction}
In optimal control theory finding a feedback law enables us to get a robust online control for dynamical systems. One prominent approach to find an optimal feedback law is calculating the value function, 
which can be done by solving either the Bellman equation or the Hamilton-Jacobi-Bellman equation. Popular numerical solutions to this problem are semi-Lagrangian methods \citep{SemiLagranigian,SemiLagrangianStochastic,Falcone1987}, Domain splitting algorithms \citep{FALCONESplitting}, variational iterative methods \citep{VIM}, data based methods with Neural Networks \citep{DataHJB} or Policy Iteration with Galerkin ansatz \citep{Beard,pol_approx_kunisch}.

The dynamical systems under consideration are spatial discretizations of non linear parabolic partial differential equations (PDE). The dimension of the HJB and the Bellman equation equals the size of the spatial discretization of the PDE, which in theory is infinite and in practice is extremely high.
Two principal difficulties are the non linearity and that it may be posed in high spatial dimensions. One can address the non linearity by the Policy Iteration. Fixing a policy reduces the non linear Bellman equation to a linear equation. The policy is updated by an optimality condition. The solution of the linearized Bellman is the remaining numerical bottleneck and is equivalent to a linear hyperbolic PDE, a linearized HJB equation.
We stress that in order to synthesize the optimal feedback control a function representation of the value function is needed.
Indeed, it is not sufficient to have point-evaluations of the value function.

Most methods for the numerical solution of hyperbolic PDEs struggle with the curse of dimensionality, i.e. the exponential growth of complexity with respect to the dimension of the underlying dynamical system. To alleviate this problem different methods have been proposed, like combinations of Proper Orthogonal Decomposition (POD) and semi-Lagrangian methods \citep{FalconePOD}, POD and tree structures \citep{ALLA2020192}, efficient polynomial Galerkin approximation and model reduction \citep{pol_approx_kunisch} or, recently, tensor based approaches \citep{horowitz2014linear, TensorKunisch}.

By using the Koopman/composition operator our approach transforms the linearized Bellman equation into an operator equation.
A solution is approximated via a Least-Squares/minimal residual method on a finite dimensional ansatz space, e.g. multi dimensional polynomials.
Even for ordinary differential equation (ODE) systems with considerably few variables this ansatz space becomes huge.
By solving the Least-Squares problem on the non linear manifold of Tensor Trains with fixed ranks we reduce the complexity  from an exponential to a polynomial dependency on the dimensions of the underlying system.
Tensor trains are particular cases of hierarchical (Tucker) or tree-based tensor representations \citep{Hackbusch-buch,Bachmayr-Uschmajew-Schneider}.
In principle, the Least-Squares method requires the evaluation of high-dimensional integrals, which is practically infeasible.
Therefore, the integration must be replaced by numerical quadrature.
Monte-Carlo and quasi Monte-Carlo methods are a canonical choice, since they do not suffer from the curse of dimensionality.
We call the resulting discrete approach variational Monte-Carlo (VMC) \citep{VMC}.
It turns out that the Koopman operator can be evaluated point-wise in certain quadrature points by computing trajectories of the underlying dynamical system.
In contrast to the linearized HJB the linearized Bellman equation can be solved model-free, i.e. without explicit knowledge of the underlying dynamical system. Finally, we remark that the present treatment of high-dimensional operators in the tensor setting differs essentially from direct treatments as in  \cite{Bachmayr-Uschmajew-Schneider, Hackbusch-buch}.
In there, the underlying partial differential operator has an explicit representation or approximation in a low rank tensor form. However, this compromises the class of treatable problems
severely. By using VMC we circumvent such a representation. 

Analogously to \cite{kundu} one can also incorporate control constraints in terms of projection operators. The generalization of the present approach to 
stochastic control problems and finite horizon problems  is relatively straightforward and should be discussed in forthcoming papers. In particular, there is a stochastic counterpart of the deterministic Koopman operator, which is the semi group generated by the backward Kolmogorov operator \citep{PerronFrobeniusKoopmanSchuette, LasotaMackey, KoopmanofRandomDynSys}.  We would like to mention recent groundbreaking progress in the treatment of backward Kolmogorov equations by means of 
deep neural networks by \citep{grohs-jensen} and further papers of these authors. 
One reason we have restricted ourselves to the deterministic  setting is that we can compare our predicted costs with computable reference values.

The treatment of the linearized Bellman equation is the same as in dynamical programming with continuous states and controls. We highlight that the adjoint of the Koopman operator corresponds to the matrix of transition probabilities in Dynamical Programming/Markov decision processes \citep{Bertsekas}. 


As in \citep{TensorKunisch}, the proposed approach uses a tensor product ansatz to find a good polynomial approximation of low
computational complexity and the linearized HJB equation is solved directly by a Galerkin (instead of Least-Squares) approximation and the high-dimensional integration is performed by TT-cross algorithms (instead of Monte-Carlo methods).

Let us remark that the present approach and particularly  the variational Monte-Carlo (or Least-Squares/minimal residual)
method  can be used for other tools in high-dimensional approximation as well. In fact, this approach coincides with the empirical risk minimization (ERM) in statistical learning \citep{steinwart}. Instead of tensor products, established methods from statistical learning can be applied to treat the linearized Bellman equation, for example (polynomial) kernel methods (support vector machine - SVM) \citep{steinwart,schoelkopf}, deep neural networks \citep{goodfellow} from machine learning, 
or sparse polynomials \citep{azmi2020data}, which are also used in uncertainty quantification \citep{schwabetal}.

We perform several numerical tests to demonstrate the numerical evidence for certain model problems. We consider a destabilized viscous Burgers equation and a diffusion equation with unstable reaction term. We compare the performance of our controller with a simple linear quadratic controller and an optimal open-loop controller for different initial states.

The rest of the paper is organized as follows.
In the preliminaries section $2$ we introduce the optimal control problem, the Koopman operator and the linearized Bellman equation.
Section $3$ is devoted to the Policy Iteration in a minimal residual treatment.
We consider the discretization of the minimal residual problem and formulate it in Monte-Carlo quadrature.
First error estimates on the variational Monte-Carlo approximation are formulated.
In Section $4$ and $5$ we introduce the manifold of tensor trains with fixed TT-rank as our ansatz space and specify the numerical routines that are used to optimize on this set.
In the final section we describe our numerical experiments.

\section{Optimal control and the Bellman equation}
In order to find an optimal feedback law, we calculate the value function of our optimal control problem. 
We show that the Bellman equation has a natural representation as an operator equation via the Koopman operator in a very general context.
\subsection{Control problem setting}
We consider a (potentially high-dimensional) ODE in $\mathbb R^d$. For $\gamma \geq 0$ we want to minimize the cost functional over $u \in  L^2([0,\infty), \mathbb R^m)$
\begin{align}
    \label{eq::infinitehorizon}
    \mathcal{J}(x,u) &= \int_{0}^{\infty} e^{-\gamma t}\big[c(y(t)) + u(t)^T B u(t) \big] dt,  \\
    \dot{y}(t) &=  f(y(t),u(t)),\label{eq::dynsyst} \\
    y(0) &= x \notag
\end{align}
where $c:\mathbb R^d \to \mathbb{R}_{\geq 0}$ is some non negative, coercive,  Lipschitz continuous cost function  with $c(0) = 0$ and $B\in \mathbb R^{m,m}$ is positive definite. For initial data $x \in \mathbb R^d$ and fixed control $u(\cdot)$, we denote by $y^x(t, u) \in \mathbb R^d$ the evaluation of the trajectory at time $t$. If the context is clear we just write $y(t)$. We further assume, that $f$ can be written as 
$$f(y(t), u(t)) = \tilde{f}(y(t)) +g(y(t)) u(t),$$
where $\tilde{f}: \mathbb R^d \to \mathbb R^d$ and $g: \mathbb R^d \to \mathbb R^{d,m}$ are smooth (non linear) functions.
We define the  value function as infimum of the cost functional over all controls 
\[ v^*(x):=\inf_{u \in L^2([0,\infty), \mathbb R^m)} \mathcal J(x, u) \text{ for all } x \in \mathbb R^d. \]
For a feedback law $\alpha: \mathbb R^d \to \mathbb R^m$ we denote the closed-loop system as
\begin{equation}\label{eq:closed_loop}
    \dot y = \tilde f(y)+g(y)\alpha(y), \quad y(0) = x.
\end{equation}
We define for fixed $\tau$ the flow $\Phi_\tau^\alpha:\mathbb R^d\to\mathbb R^d$ such that $\Phi_\tau^\alpha(x)$ is the evaluation of the trajectory at time $\tau$ with initial state $x$ w.r.t \eqref{eq:closed_loop}. In the following, we refer to the case $\gamma > 0 $ as the discounted and to $\gamma = 0$ as the undiscounted problem. We write 
\[r(y(t),u(t)) := c(y(t)) + u(t)^T B u(t).\]
\vspace{-.5cm}

\begin{lemma}\label{lem:flow}
Let $\alpha:\mathbb R^d\to \mathbb R^m$ be Lipschitz. Then the flow $\Phi_\tau^\alpha$ of the dynamical system \eqref{eq:closed_loop} is well defined and Lipschitz continuous.
\end{lemma}

In the following we only consider feedback laws that give us finite costs. Thus, we define the policy evaluation function $v^\alpha$ with respect to a feedback law $\alpha$ as
\begin{equation}\label{eq:v_alpha}
    v^\alpha(x) :=   \int_{0}^{\infty} e^{-\gamma t}\big[c(\Phi_t^\alpha(x)) + \alpha(\Phi_t^\alpha(x))^T B \alpha(\Phi_t^\alpha(x)) \big] dt.
\end{equation} 
Then we can define the set of feedback laws that induce finite costs 
$$F = \{ \alpha:\mathbb R^d \to \mathbb R^m\ |\ v^\alpha(x)<\infty \text{ for all } x \in \mathbb R^d \text{ and } \alpha \text{ Lipschitz}\}.$$
The following assumptions ensure that there exists a feedback law, where the policy evaluation function is the value function and that $v^*(0) = 0$.
	\begin{enumerate}[label=\textbf{(A.\arabic*)}]
		\item There exists an optimal Lipschitz continuous feedback control with finite costs for any initial state $x\in \mathbb R^d$. \label{assum_2}
		\item For all $\alpha \in F$ we assume $v^\alpha $ is $C^1(\mathbb R)$ with Lipschitz derivative.
		\item The origin is a steady state of the uncontrolled dynamical system, i.e. $\tilde f(0)=0$.
	\label{as:assum_2}
	\end{enumerate}

\begin{remark}
The assumption that there exists an optimal Lipschitz feedback control is very strong. The works \citep{Gota,RINCONZAPATERO2012305, BREITEN20191361,BreitenStokes} give some examples, when this is the case. In \cite{schaft} local smoothness of the value function and policy evaluation function for smooth systems is shown.
We use these assumptions to ensure that the flow is well-defined.
If this can be guaranteed in another way, these assumptions can be dropped.
\end{remark}
\subsection{The Bellman equation}
The value function obeys the Bellmann equation.
\begin{theorem} Consider $x\in \mathbb R^d$. Then for any $\tau\in [0,\infty)$ we have:
\begin{equation}\label{BellmannsPrinciple}
v^*(x) = \min_{\alpha\in F} \lbrace \int_0^\tau e^{-\gamma t}r(\Phi_t^\alpha(x), \alpha(\Phi_t^\alpha(x))) dt + e^{-\gamma \tau}v^*\circ \Phi_\tau^\alpha(x)\rbrace.
\end{equation}
\end{theorem}
The Bellman equation holds in a very general context, which includes our setting \citep{DaLio,DeterministicHJB,Barbu,LiYongOptContrInfDim}.
The flow of a dynamical system allows us to lift the dynamics from the state space to the space of all functionals on the state space via a family of Koopman operators $K_t^\alpha[h](x) = h\circ \Phi^\alpha_t(x)$ for $h$ in some appropriate function class \citep{PerronFrobeniusKoopmanSchuette,AppliedKoopmanism,Koopman315}. 
Defining $r^\alpha(x) := r(x, \alpha(x))$ and $r^\alpha_t (x)= r^\alpha(\Phi_t^\alpha(x))$, the Bellman equation reads
\[ 0 = \min_{\alpha\in F} \lbrace\int_0^\tau e^{-\gamma t}r^\alpha_t(\cdot)dt + (e^{-\gamma \tau}K_\tau^\alpha-\id)[v^*](\cdot)\rbrace.\]
Notice, that for fixed feedback law $\alpha\in F$ the corresponding cost functional $v^\alpha$ satisfies a linearized Bellman-type equation, i.e.
\begin{equation}\label{eq:linear-Bellmann} 
0 = \int_0^\tau e^{-\gamma t}r^\alpha_t(\cdot)dt + (e^{-\gamma \tau}K_\tau^\alpha-\id)[v^\alpha](\cdot).
\end{equation}

\begin{remark}
The advantage of the Koopman operator is that it allows to write the linearized Bellman equation as fix-point equation. In upcoming works we apply the framework of \citep{Puterman} to the convergence analysis of this continuous time, continuous state optimal control problem. 
\end{remark}

We introduce the coupled Bellmann-like equations, i.e.
\begin{align*}
         0 &= \int_0^\tau e^{-\gamma t}r^\alpha_t(x)dt + (e^{-\gamma \tau}K_\tau^\alpha-\id)[v](x),\\
  & \alpha^*(x) = -\frac{1}{2}B^{-1} g(x)^T \nabla v^*(x) .
\end{align*}

In \citep{TensorKunisch} one target is to solve the HJB \citep{DeterministicHJB} \begin{equation} \label{eq:DeterministicHJB}
  \min_{\alpha\in F} \{-\gamma v^*(x)+\nabla v^*(x)\cdot f(x,\alpha(x))+r(x,\alpha(x))\}=0.  
\end{equation}
The authors of \citep{TensorKunisch} solve this equation with policy iteration. There within they need to approximate for fixed policy $\alpha$
\begin{equation} \label{eq:linHJB}
  -\gamma v^*(x)+\nabla v^*(x)\cdot f(x,\alpha(x))+r(x,\alpha(x))=0.  
\end{equation}
using a Galerkin method.
For their algorithm it is necessary to evaluate the right hand side of the ODE $f(x, \alpha(x))$ at certain points $x \in \mathbb R^d$. In contrast to that, we take a first steps towards a model-free approach, such that we can solve optimal control problems resulting from experiments or simulations of research partners with unknown discretization of the controlled dynamical system. Thus, we transform the linearized system \eqref{eq:linHJB} using the method of characteristics to obtain an equation where a black-box solver of the ODE and the reward function $r$ can be used. In the end, the only a priori information we need are the control functions $g$ and the matrix $B$. Note that as $B$ is constant and linear it can easily be learned from data. Moreover, in many applications $g$ is constant w.r.t. $x$, again making it easy to learn it. Both of these steps can be done even before finding an optimal policy.

We finish the section by informally stating the Policy Iteration algorithm.

\begin{algorithm}[H]
\SetAlgoLined
\caption{Policy Iteration.}\label{algo:pol_it_informal}
\SetKwInOut{Input}{input}\SetKwInOut{Output}{output}
\SetKwInOut{Output}{output}\SetKwInOut{Output}{output}
\Input{Initial Policy $\alpha_0$}
\Output{An approximation of $v^*$ and $\alpha^*$, denoted by $\hat v$ and $\hat \alpha$}

Set $k = 0$.

\While{not converged}{
  Solve for $v_k$ the linear equation
\begin{equation}\label{eq:PolicyIteration1}
     0 = \int_0^\tau e^{-\gamma t}r^{\alpha_k}_t(\cdot) dt + (e^{-\gamma \tau}K_\tau^{\alpha_k}-\id)v_k(\cdot)
\end{equation}
then update the policy according to
\begin{equation}\label{eq:policy_update}
 \alpha_{k+1}(x) = -\frac{1}{2}B^{-1} g(x)^T \nabla v_k(x) .
\end{equation}
Set $k = k+1$.
}
Set $\hat v = v_{k}$ and $\hat \alpha = - \frac 1 2 B^{-1}  g^T \nabla \hat v $.
\end{algorithm}
\begin{remark}
The Koopman operator $K_\tau^\alpha$ and its infinitesimal generator $\mathcal L^\alpha$ reveal the operator theoretic interplay of the linearized HJB and Bellman equation. In the linear-quadratic case the policy iteration with \eqref{eq:linHJB} instead of \eqref{eq:PolicyIteration1} corresponds to the Newton-Kleinman method \citep{kleinman}. In the non-linear case with \eqref{eq:linHJB} instead of \eqref{eq:PolicyIteration1} there are convergence results by \cite{Puterman, Saridis}. Hence, if \eqref{eq:PolicyIteration1} is well defined in every step, this algorithm converges. However, the convergence for approximative schemes and an error analysis are only partially answered. In \citep{DeterministicHJB} the use of monotone schemes was analysed.  In \citep{Beard} Galerkin approximation was analyzed. Least-squares methods are still open as is an error analysis.
\end{remark}

\section{Approximative Policy Iteration}
To numerically solve the optimal control problem we do not calculate the value function on $\mathbb R^d$ but instead approximate it on a compact $\Omega \subset \mathbb R^d$. 
To this end, we first formulate our procedure informally.

We consider the Hilbert space $ V := L^2 (\Omega ) $, (or more generally $ V := L^2 ( \Omega, \rho ) $ with some probability density $ \rho$ and fix some $\alpha \in F$ such that the exact solution $v^{\alpha}$ to the linearized Bellman equation 
\[ 0 = \int_0^\tau e^{-\gamma t}r^{\alpha_k}_t(\cdot) dt + (e^{-\gamma \tau}K_\tau^{\alpha_k}-\id)v_k(\cdot) \]
is element of $V$.
Denote by $$\mathcal R^\alpha(v) = \|(\id -e^{-\gamma\tau}K_\tau^\alpha)v-\int_0^\tau e^{-\gamma t}r_t^\alpha(\cdot)\  dt\|^2_{V}.$$
Then $ v^{\alpha} \in V $ is a minimizer of the functional $ \mathcal{R}^{\alpha} $, i.e.
\begin{equation*}
    v^{\alpha} = \argmin_{ v \in V} \mathcal{R}^{\alpha} (v).
\end{equation*}
In particular, we have $\mathcal R^\alpha(v^\alpha) = 0$. However, finding this minimizer is infeasible and we further restrict to a subset.
Let $ U \subset V$, the {\em minimal residual approximation} (or Least-squares approximation) is defined by the minimizer
\begin{equation*}
    v^{\alpha}_U = \argmin_{ v \in U} \mathcal{R}^{\alpha}(v).
\end{equation*}
Classically, $U \subset V$ is a finite dimensional subspace. However, in many applications the subspace $U$ is high-dimensional, which makes computations impracticable. Thus, we allow $\mathcal{M} \subset U$ to be a compact subset
feasible for computational treatment having an intrinsic data complexity, which can be handled
by our technical equipment. In our case $\mathcal{M}$ is the set of tensor trains of bounded (multi-linear) rank
$\mathbf{r} = (r_1, \dots, r_n)$ and uniformly bounded norm embedded in the space $U\subset V$ of tensor product polynomials of degree $p$. We will introduce this function class in Section \ref{sect:TT}.

However, the numerical treatment of the above 
minimization problem is still infeasible due to the presence of the high-dimensional integrals over $\Omega \subset \mathbb{R}^d$. To handle this problem  we replace the exact integral by a numerical quadrature, e.g.
uniformly distributed Monte-Carlo integration.
\begin{equation*}
    v^{\alpha}_{(\mathcal{M}, N) } \in \argmin_{ v \in \mathcal M} \mathcal{R}_N^{\alpha}(v) , \ 
  \mathcal R_N^\alpha(v) =  \frac{1}{N} \sum_{i=1}^N |(\id -K_\tau^\alpha)v(x_i)-\int_0^\tau r_t^\alpha(x_i)dt|^2 .  
\end{equation*}
Note that this minimization problem is computable now because the Koopman operator can be evaluated pointwise using a black-box ODE solver.
\begin{remark}
Recall that in contrast to our method in \citep{TensorKunisch} the linearized HJB \eqref{eq:linHJB} is solved directly using a Galerkin scheme. The numerical integration is done by the TT-cross algorithm. 
\end{remark}

In the following we consider the approximative Policy Iteration

\begin{algorithm}[H]
\SetAlgoLined
\caption{Approximative Policy Iteration.}\label{algo:pol_it}
\SetKwInOut{Input}{input}\SetKwInOut{Output}{output}
\SetKwInOut{Output}{output}\SetKwInOut{Output}{output}
\Input{Initial Policy $\alpha_0\in F$}
\Output{An approximation of $v^*$ and $\alpha^*$, denoted by $\hat v$ and $\hat \alpha$}
Set $k = 0$.

\While{not converged}{
  Solve the linear equation
\begin{equation}\label{PolicyIteration1}
     v_k \in \argmin_{v\in \mathcal M} \mathcal R_N^{\alpha_k}(v)
\end{equation}
then update the policy according to
\begin{equation}\label{policy_update}
   \alpha_{k+1}(x) = -\frac{1}{2}B^{-1} g(x)^T \nabla v_k(x).
\end{equation}
Set $k = k+1$.
}
Set $\hat v = v_{k}$ and $\hat \alpha = - \frac 1 2 B^{-1}  g^T \nabla \hat v $.
\end{algorithm}
\begin{remark}
    Solving the linearized Bellman equation inexactly breaks the arguments of the convergence proofs in \cite{Puterman,Saridis}. A generalization of the convergence proofs to an approximative form as in Algorithm \ref{algo:pol_it} is non trivial and part of future research.
    However, as $\mathcal M$ consists of smooth functions, the well-definedness of the flow is guaranteed.
\end{remark}
Due to the randomness of the Monte-Carlo quadrature, we cannot expect uniform or worst case error estimates. Instead, one can show 
under appropriate assumptions that an error estimate holds with high probability and decays exponentially with the number of quadrature points $N$,
see e.g. \citep{steinwart,VMC} and references therein.

If we have $\Omega\subset \Phi^\tau_\alpha(\Omega)$ the Koopman operator $K_\tau^\alpha:L^p(\Omega)\to L^p(\Omega)$ is a bounded linear operator, see \citep{CompositionOpLp}. Together with the differentiability of our cost functional,  we can give some bounds on errors in probability \citep{VMC}.

\begin{proposition}\label{prop:vmc} 
Let $\varepsilon>0$
. Then 
\[ \mathbb P[|\mathcal R^\alpha(v_{\mathcal M}^\alpha)-\mathcal R^\alpha(v^\alpha_{(\mathcal M,N)})|> \varepsilon] \leq c_1(\varepsilon, \mathcal M, \alpha) e^{-c_2(\alpha) N \varepsilon^2} \]
with $c_1(\varepsilon, \mathcal M, \alpha), c_2(\alpha)>0$.
\end{proposition}
\begin{proof}
We can apply Theorem 4.12 and Corollary 4.19 from~\citep{VMC}.
\end{proof}
We shortly comment on the behavior of the constant $c_1(\varepsilon, \mathcal M)$ from Proposition \ref{prop:vmc}. To this end we have to analyze the behavior of the constant $c_1$ with respect to the dimension of the ambient space, which is in this case the tensor product of one-dimensional polynomials $U$, where the set of tensor trains $\mathcal M$ is embedded into. Note that the constant $c_1$ is proportional to the covering number of the considered set.
The covering number $\nu(\mathcal M, \varepsilon)$ of a subset $\mathcal M \subset U$ is defined as the minimual number of $\| \cdot \|_{L^2(\Omega)}$-open balls of radius $\varepsilon$ needed to cover $U$. For the open ball in the full ambient space $U$ this can be estimated as $\nu(B_1(0), \varepsilon) = \mathcal O((1/\varepsilon)^{p^n})$. However, if we intersect the open ball with the tensor manifold of fixed ranks, the covering number can be estimated as $\mathcal O ((n/\varepsilon)^{p r^2 n})$ \citep[Lemma 3]{RAUHUT2017220}. Thus, Proposition \ref{prop:vmc} proves a $\mathcal O(\log(n) n)$ dependence of the number of samples $N$ on the dimension $n$ of the underlying space.

\begin{remark}
\label{rem:bounded_inverse}
Proposition \ref{prop:vmc} measures the error in terms of $\mathcal R^\alpha$. Note that under stronger assumptions one can give an estimate in terms of $\|\cdot \|_{L^2(\Omega)}$, i.e.

\[ \mathbb P[\|v^\alpha-v^\alpha_{(\mathcal M,N)}\|_{L^2(\Omega)}> \varepsilon] \leq c_1(\varepsilon, \mathcal M) e^{-c_2(\alpha) N \varepsilon^2}. \]
Further analysis of the Koopman operator is necessary to show this bound, in particular that there exists a continuous inverse of $I - K_\tau^\alpha$ in an appropriate function space. Note that choosing the discount $\gamma$ such that $e^{-\gamma\tau} \|K_\tau^\alpha\|_{\mathcal L(L^p(\Omega))}<1$ guarantees the existence of a bounded inverse by the Neumann series \citep{CompositionOpLp}.
\end{remark}
\begin{remark}

We call the discretization described above Variational 
Monte-Carlo (VMC) \citep{VMC}, a terminology 
introduced in 
computational physics. The terminologies 
{\em discrete minimal residual method} or
{\em discrete least squares approximation}
are also appropriate. 
An improved analysis, in particular for
low rank tensor approximations, has been considered
in the paper \cite{eigel2020convergence}.
However the theory is still in its infacy.

Let us highlight the analogy of VMC with 
{\em empirical risk minimization (ERM)} in machine learning \citep{steinwart}. Indeed, ERM is a particular example for VMC. Moreover, many models in 
statistical learning can be used instead, e.g.
(deep) neural networks. The present approach is a numerical discretization method like collocation or Nystr\"om method. 

\end{remark}

\section{Tree Based Tensor Representation - Tensor Trains}\label{sect:TT}
For the approximation of the value function we define a non linear model class to circumvent the curse of dimensionality. 
To this end we choose an underlying finite dimensional subspace for the approximation of the sought value function. For the present purpose we 
take a space $\Pi_{i,n_i} = \spa \{\psi_{i_1}, \dots, \psi_{i_d} \}$ of one-dimensional polynomials of degree smaller than $n_i$  and consider the tensor product of such polynomial spaces
$$ \mathcal{V}_p := \Pi_{1,n_1} \otimes \cdots \otimes \Pi_{d,n_d}.
$$
This is a space of multivariate (tensor product) polynomials with bounded multi-degree.
Its elements $v \in \mathcal V_p$ can be represented as
\begin{equation}\label{eq:function_fulltensor}
v ( x_1 , \ldots , x_d ) = 
\sum_{i_1 , \ldots , i_d = 0}^{n_1, \dots, n_d}
c_{ i_1, \ldots , i_d} \psi_{i_1} (x_1) \cdots 
\psi_{i_d} (x_d),
\end{equation}
exhibiting that $c\in \mathbb R^{n_1, \dots, n_d}$ suffers from the curse of dimensionality.
For the sake of readibility we will henceforth write $c[i_1, \dots, i_d] = c_{i_1, \dots, i_d}$ and say that $c$ is an order $d$ tensor.

Using structured representations of polynomials like hierarchical tensor formats allows to reduce the number of parameters within the tensor $c$ \citep{Hackbusch-buch}.
More exactly, we consider a sub-manifold in $\otimes_{j=1}^d \mathbb{R}^{n_i}$
defined by multi-linear parametrizations.  
Here we use tensor trains which are a special case of a hierarchical or tree based tensor format \citep{Hackbusch-buch}. Tensor trains have been invented 
by \citep{Oseledets,Oseledets2} and applied to various high-dimensional PDE's \citep{Khoromskij-book}, but the parametrization has been used in quantum physics much earlier 
as {\em Matrix Product States} and {\em Tensor Network States} successfully for the approximation of spin systems and Hubbard model. 
For good  surveys  we refer to 
\citep{Hackbusch2014,Bachmayr-Uschmajew-Schneider,Legeza-Schneider,Hackbusch-Acta}. 
The tensor train representation  have appealing properties
making them attractive for treatment of the 
present problems, compare \citep{TensorKunisch}. For example they contain sparse polynomials, 
but are much more flexible at a price of a slightly larger overhead, see e.g. \citep{Dahmen-Bachmayr} for a comparison concerning parametric PDEs. 

The tensor train decomposition aims to represent an order $d$ tensor by a sequence of order $3$ tensors, connected by contractions.
This means that we represent $c$ by $U_1 \in \mathbb R^{n_1, r_1}$, $U_2 \in \mathbb R^{r_1, n_2, r_2}, \dots, U_{d-1} \in \mathbb R^{r_{d-2}, n_{d-1}, r_{d-1}}$ and $U_d \in \mathbb R^{r_{d-1}, n_d}$ such that
\begin{equation}\label{eq:TT_representation_coefficienttensor}
    c[i_1, \dots, i_d] = \sum_{j_1 = 1}^{r_1} \dots \sum_{j_{d-1} = 1}^{r_{d-1}} U_1[i_1, j_1] U_2 [j_1, i_2, j_2] \dots U_d[j_{d-1}, i_d].
\end{equation}
The TT-rank is introduced as the element wise smallest tuple $\mathbf r = (r_1, \dots, r_{d-1})$ such that a decomposition of the form \eqref{eq:TT_representation_coefficienttensor} exists.
The TT-rank is well defined and,
denoting $r = \max \{r_i\}$ and $n = \max \{n_i\}$,
the tensors of fixed TT-rank form a smooth manifold of dimension in $\mathcal O (dnr^2)$ \citep{TTTensor}, which means that for fixed ranks the dimension of the manifold does increase linearly with the order $d$.
Furthermore, allowing for tensors with smaller TT-rank one obtains an algebraic variety, denoted by $\mathcal M$ \citep{landsberg2012tensors, KUTSCHAN}.
Note that in particular $\mathcal M$ is closed. 
However, the numerical routines do not differentiate between the variety and the manifold. We use this approach to tackle the curse of dimensionality.
Observing that the component tensors $U_i$ are connected via a single contraction/summation to $U_{i-1}$ and $U_{i+1}$, we can represent the decomposition in a graph, by setting the components $U_i$ as nodes and indicate contractions by links between the nodes.
\begin{figure}[h!]
    \centering
    \begin{tikzpicture}
\begin{scope}[every node/.style={scale=1,minimum size=5mm, draw,  fill=white}]
    \node (A1) at (0,0) {$U_1$}; 
    \node (A2) at (2,0) {$U_2$}; 
    \node (A3) at (4,0) {$U_3$}; 
    \node (A4) at (6,0) {$U_4$}; 
    \node (C) at (-4,0) {$c$}; 
\end{scope}
    \node (D) at (-2,0) {$=$}; 
\begin{scope}[every edge/.style={draw=black,thick}]
	\path [-] (A1) edge node[midway,left] [above] {$r_1$} (A2);
	\path [-] (A3) edge node[midway,left] [above] {$r_2$} (A2);
	\path [-] (A3) edge node[midway,left] [above] {$r_3$} (A4);
	\path [-] (A1) edge node[midway,left] {$n_1$} (0,-1);
	\path [-] (A2) edge node[midway,left] {$n_2$} (2,-1);
	\path [-] (A3) edge node[midway,left] {$n_3$} (4,-1);
	\path [-] (A4) edge node[midway,left] {$n_4$} (6,-1);
	\path [-] (C) edge node[midway,left] [above] {$n_1$} (-3,0);
	\path [-] (C) edge node[midway,left] {$n_2$} (-4,1);
	\path [-] (C) edge node[midway,left] [below] {$n_3$} (-5,0);
	\path [-] (C) edge node[midway,left] [right]{$n_4$} (-4,-1);
\end{scope} 
\end{tikzpicture}
    \caption{Graphical representation of a TT representation of $c$ in four variables.}
    \label{TT_polynomial}
\end{figure}
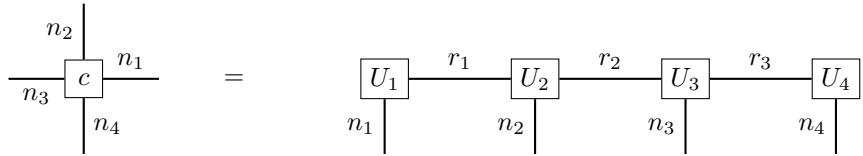
In the next step we plug the TT-decomposition of the coefficient tensor \eqref{eq:TT_representation_coefficienttensor} into the representation in \eqref{eq:function_fulltensor}.
    To this end we introduce the short form $\Psi_i(x_i) = [\psi_{i_1}(x_i), \dots, \psi_{i_d}(x_i)] \in \mathbb R^{n_i}$.
Then 
\begin{multline*}
    v(x_1,\dots,x_n) =\sum_{i_1,\dots,i_d}^{n_1,\dots,n_d} \sum_{j_1,\dots,j_{n-1}}^{r_1,\dots,r_{n-1}} U_1[i_1,j_1] U_2[k_1, i_2, k_2] \dots U_d[j_{d-1},i_d] \big(\Psi_1(x_1)\big)[i_1](\Psi_2(x_2)\big)[i_2]\cdots \big(\Psi_n(x_n)\big)[i_n],
\end{multline*}
which means that every open index of the TT representation is contracted with the one-dimensional basis functions.
The graphical representation of this tensor network is given in Figure \ref{TT_polynomial}.
\begin{figure}[h!]
    \centering
    \begin{tikzpicture}
\begin{scope}[every node/.style={scale=1,minimum size=5mm, draw,  fill=white}]
    \node (A1) at (0,0) {$U_1$}; 
    \node (A2) at (2,0) {$U_2$}; 
    \node (A3) at (4,0) {$U_3$}; 
    \node (A4) at (6,0) {$U_4$}; 
	\node [ position=-90:1 from A1](B1) {$\Psi_1(x_1)$};
   	\node [ position=-90:1 from A2](B2) {$\Psi_2(x_2)$};
   	\node [ position=-90:1 from A3](B3) {$\Psi_3(x_3)$};
   	\node [ position=-90:1 from A4](B4) {$\Psi_4(x_4)$};

\end{scope}
    \node (C) at (-4,0) {$v(x)$}; 
    \node (D) at (-2,0) {$=$}; 

\begin{scope}[every edge/.style={draw=black,thick}]
	\path [-] (A1) edge node[midway,left] [above] {$r_1$} (A2);
	\path [-] (A3) edge node[midway,left] [above] {$r_2$} (A2);
	\path [-] (A3) edge node[midway,left] [above] {$r_3$} (A4);
	
	\path [-] (A1) edge node[midway,left]  {$n_1$} (B1);
	\path [-] (A2) edge node[midway,left]  {$n_2$} (B2);  
	\path [-] (A3) edge node[midway,left]  {$n_3$} (B3);  
	\path [-] (A4) edge node[midway,left]  {$n_4$} (B4);  
\end{scope} 
              
\end{tikzpicture}
    \caption{Graphical representation of TT tensor train induced polynomial in four variables.}
    \label{TT_polynomial}
\end{figure}
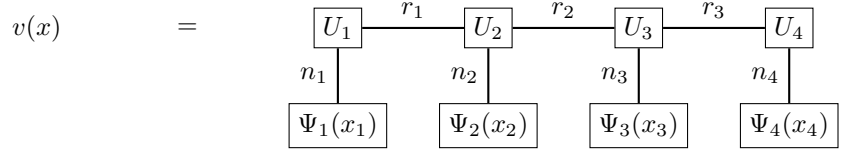
Note that any basis can be chosen for $\Psi_i$. 
In this paper we use a set of orthonormal polynomials.
In this case,  we have a Parseval formula providing a norm equivalence between the function space and the Frobenius norm of the coefficients, which guarantees stability of our representations

It turns out, that optimization procedures in this TT format can be solved by consecutively optimizing one component $U_l$ while the others are fixed. This alternating Least-Squares  (ALS) algorithm converges at least to a local minimum \citep{ALS}. To that end, we reorder the coefficients of $v$ with respect to the index $l$, i.e. we gather every index smaller than $l$ and every index larger than $l$,
\begin{align}
    v(x_1,\dots,x_n) =&\sum_{i_l}^{n_l} \sum_{j_{l-1},j_l}^{r_{j-1},r_j} U_j[j_{l-1},i_l,j_{l}]\Psi_l(x_l)_{i_l} \notag \\ 
    &\Big(\sum_{i_1, \dots i_{l-1}}^{n_1, \dots, n_{l-1}}
    \sum_{j_1, \dots, j_{l-2}}^{r_1, \dots, r_{l-2}} U_1[i_1,j_1] \dots U_{l-1}[j_{l-2},i_l, j_{l-1}]\big(\Psi_1(x_1)\big)[i_1] \dots (\Psi_{l-1}(x_{l-1})\big) [i_{l-1}]
    \label{eq:local_basis1}\\
    &\sum_{i_{l+1}, \dots i_d }^{n_{l+1}, \dots n_d} 
    \sum_{j_{l+1}, \dots, j_{d-1}}^{r_{l+1}, \dots, r_{d-1}}
    U_{l+1}[j_l, i_{l+1},j_{l+1}] \dots U_{d}[j_{d-1},i_d]\big(\Psi_{l+1}(x_{l+1})\big)[i_{l+1}] \dots (\Psi_{d}(x_{d})\big) [i_{d}]
    \Big) \label{eq:local_basis2}\\
    =&\sum_{i_{l-1},i_{l}}^{n_{l-1},n_{l}} \sum_{j_{l-1},j_{l}}^{r_{l-1},r_{l}} U_l[j_{l-1},i_l,j_{l}]\big(\Psi_l(x_l)\big)[i_l] \big(\tilde{b}^L(x_{i<l})\big) [j_{l-1}]\big(\tilde{b}^R(x_{i>l})\big)[j_{l}]. \notag
\end{align}
Note that $\tilde{b}^L(x_{i<j})\in \mathbb R^{r_{l-1}}$ represents \eqref{eq:local_basis1}, i.e. every index smaller than $l$ and $\tilde{b}^R(x_{i>j})\in \mathbb R^{r_{l}}$ represents \eqref{eq:local_basis2}, i.e. every index larger than $l$.

By doing this reordering, we see that $U_l$ can be seen as the coefficients of local basis functions that are given as the (tensor-) product of $\Psi_l$, $\tilde b^L$ and $\tilde b^R$, which form a linear space  $\mathcal V_{\text{loc}}$ of dimension $r_{l-1} \cdot n_l \cdot r_l$ . The ALS algorithm makes use of this observation, by consecutively identifying the linear spaces $\mathcal V_{\text{loc}}$.

For further clarification, we have added the graphical representation of the basis functions in Figure \ref{fig:local_basis} for the case $l=2$.
\begin{figure}[h!]
    \centering
    \begin{tikzpicture}
\begin{scope}[every node/.style={scale=1,minimum size=5mm, draw,  fill=white}]
    \node (A1) at (0,0) {$U_1$}; 
    \node (A3) at (9,0) {$U_3$}; 
    \node (A4) at (11,0) {$U_4$}; 
	\node [ position=-90:1 from A1](B1) {$\Psi_1(x_1)$};
    \node (B2) at (4,-1) {$\Psi_2(x_2)$}; 
   	\node [ position=-90:1 from A3](B3) {$\Psi_3(x_3)$};
   	\node [ position=-90:1 from A4](B4) {$\Psi_4(x_4)$};

\end{scope}
    \node (C1) at (-2,0) {$\tilde b^L(x_{i<j})$}; 
    \node (D1) at (-1,0) {$=$}; 
    \node (C2) at (2,0) {; $\Psi_l(x_l)$}; 
    \node (D2) at (3,0) {$=$}; 

    \node (C2) at (6,0) {; $\tilde b^R(x_{i>j})$}; 
    \node (D2) at (7,0) {$=$}; 

\begin{scope}[every edge/.style={draw=black,thick}]
	\path [-] (A1) edge node[midway,left] [above] {$r_1$} (1, 0);
	\path [-] (A3) edge node[midway,left] [above] {$r_2$} (8, 0);
	\path [-] (A3) edge node[midway,left] [above] {$r_3$} (A4);
	
	\path [-] (A1) edge node[midway,left]  {$n_1$} (B1);
	\path [-] (4, 0) edge node[midway,left]  {$n_2$} (B2);  
	\path [-] (A3) edge node[midway,left]  {$n_3$} (B3);  
	\path [-] (A4) edge node[midway,left]  {$n_4$} (B4);  
\end{scope} 
              
\end{tikzpicture}
    \caption{Graphical representation of the local basis functions for the case $l=2$.}
    \label{fig:local_basis}
\end{figure}
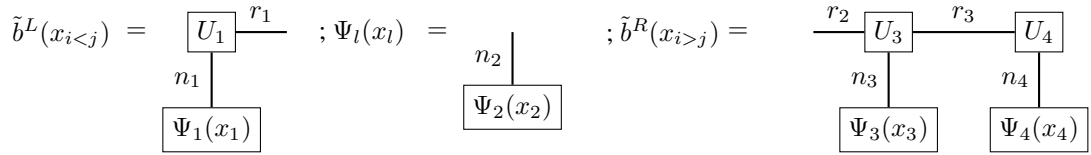

\section{Variations of the VMC equation}\label{sect:variant_vmc}
In this section we give a detailed explanation on how to solve the VMC equation \eqref{PolicyIteration1} on the manifold of Tensor Trains.
To this end, we first consider the a linear ansatz space.
We distinguish between the formulation in the function space, where we denote the loss functional by $\mathcal R_N(v)$, and the formulation in the coefficient space, where $v$ is represented by its coefficients in some appropriate basis. To this end we define $v(x) = \sum_{j = 1}^M c_j b_j(x)$, with $ b = \{ b_j \}_{j=1}^M $ is a basis of an appropriate $M$ dimensional ansatz space. We obtain an equivalent loss functional, which we denote by $\mathcal R_N(\mathbf c)$. The Least-Squares loss functional in function space is given by
\begin{align}
    \mathcal R_N(v) &= \frac 1 N\sum_{i=1}^N | v(x_i) - v(\Phi_\tau^\alpha(x_i))-\int_0^\tau r_t^\alpha(x_i)dt|^2, \notag \\
    & \approx \frac 1 N\sum_{i=1}^N |v(x_i) - v(\Phi_\tau^\alpha(x_i))-R(x_i) |^2,\label{eq:vmc_loss_functional}
\end{align} 
where $R(\cdot) \approx \int_0^\tau r_t^\alpha(\cdot) dt$ is the trapezodial or some other quadrature rule operator. The formulation of the loss functional in coefficient space is
\begin{equation}
    \label{eq::HJBCoeffMC}
    \mathcal R_N (\mathbf c) =  \sum_{i=1}^N \sum_{j=1}^M |[b_j(x_i) - b_j(\Phi(x_i)) ] c_j - R(x_i) |^2 =: \| A \mathbf c - \mathbf R \|_{\mathbb R^N}^2,
\end{equation}
with $A \in \mathbb R^{m, n}$, $a_{ij} = b_j(x_i) -b_j(\Phi(x_i))$. It is well known that the minimizer of this functional is attained by $\mathbf c$ if $A^T A \mathbf c = A^T \mathbf R$. 

We now derive some variations of this loss functional, where more information is encoded and focus on the formulation in the function space. We know that $v(0) = 0 $ and $\nabla v(0) = 0$, and thus we introduce penalty terms for constant and linear polynomials. Note that in the case of linear ansatz spaces, we could simply not include the constant and linear polynomials. However, this will later not be possible. In order to circumvent overfitting, we also add a penalty term for the norm of the value function
\[ \tilde {\mathcal R}_N(v) = \mathcal R_N(v) + \delta_1| v(0) |^2 + \delta_2| \nabla v(0)|^2 + \delta_3 \|v\|_{H^1_{\text{mix}}(\Omega)}^2. \]
where $H^1_{\text{mix}}(\Omega)$ is the tensor product of one-dimensional $H^1$ Sobolev spaces \cite{sickel2009tensor}, assuming that $\Omega$ can be written as $\Omega = \bigotimes_{i = 1}^d [a, b]$, where $a < b$.
Choosing the one-dimensional ansatz functions to be orthonormal w.r.t. $H^2(a, b)$, this regularization term is realized via Parseval's identity, by penalizing the Frobenius norm of the component tensors.
As $H^1_{\text{mix}}(\Omega)$ is continuously embedded into $L^\infty (\Omega)$ \cite{sickel2009tensor} we are penalizing the $L^\infty(\Omega)$ norm of the derivative of $v$, which prevents overfitting.

Choosing the basis $b$ to be orthonormal with respect to $H^1(\Omega)$, we can represent the last term of the loss functional in coefficient space using Parseval's identity:
\[ \|v\|_{H^1_{\text{mix}}(\Omega)}^2 = \| \mathbf c \|_F^2, \]
where $\mathbf c$ is the coefficient vector of $v$ in the basis $b$. We obtain
\begin{equation}\label{eq:least_squares_penalty}
     \tilde{\mathcal R}_N (\mathbf c)  = \| A\mathbf c - r \|^2 
     +\delta_1 \sum_{i = 1}^M (c_i b_i(0))^2
     +\delta_2 \sum_{j = 1} ^ N \sum_{i = 1} ^M (c_i \partial_{x_j} b_i(0))^2 +\delta_3 \| \mathbf c \|_F^2. 
\end{equation}
\begin{remark} 
Our cost functional is very flexible and can consist of different or additional terms. For example, one could add the absolute value of the pointwise evaluation of the HJB equation as penalty term, analogously to physical informed neural networks (PINN) \cite{RAISSI2019686}.
\end{remark}
\subsection{The VMC equation on the TT Manifold}
In the TT case with fixed ranks we do not have a linear ansatz space and thus the solution to the problem is not found as easily.
We can still fomulate $\mathcal R_N(v)$, but the coefficient representation of the loss functional demands further clarification.
The ALS algorithm solves this problem by reducing the non linear Least-Squares problem to a sequence of several small linear problems, where the above formulation is valid.
This is done by using the local basis functions, given by (\ref{eq:local_basis1}, \ref{eq:local_basis2}).
For the small problem we can use the above formulation of $\mathcal R_N(\mathbf c)$.
This is summarized in the following algorithm, c.f. \citep{ALS} for more details.
\begin{algorithm}[H]
\SetAlgoLined
\caption{Alternating Least-Squares (ALS)}\label{algo:als}
\SetKwInOut{Input}{input}\SetKwInOut{Output}{output}
\SetKwInOut{Output}{output}\SetKwInOut{Output}{output}
\Input{A TT representation with coefficient tensors $U_1, \dots, U_d$.}
\Output{Optimized coefficient tensors $U_1, \dots, U_d$.}
\While{not converged}{
    \For{$\mu = 1 \dots d-1$}{
    Identify the local basis functions as in Figure \ref{fig:local_basis} and then solve the linear equation \eqref{eq:least_squares_penalty}. 
    
    Set $U_\mu$ as the coefficients obtained by solving the linear equation.
    }
    \For{$\mu = d \dots 2$}{
    Identify the local basis functions as in Figure \ref{fig:local_basis} and then solve the linear equation \eqref{eq:least_squares_penalty}. 
    
    Set $U_\mu$ as the coefficients obtained by solving the linear equation.
    }
}
\end{algorithm}
We say that one inner for loop is a half sweep and that completing both loops is a sweep.

The algorithm is summarized as follows

\begin{algorithm}[H]
\SetAlgoLined
\caption{Basic Policy Iteration}\label{algo:basic_pol_it}
\SetKwInOut{Input}{input}\SetKwInOut{Output}{output}
\SetKwInOut{Output}{output}\SetKwInOut{Output}{output}
\Input{Initial Policy $\alpha_0\in F$}
\Output{An approximation of $v^*$ and $\alpha^*$, denoted by $\hat v$ and $\hat \alpha$}
Set $k = 0$.

\While{not converged}{
  Solve 
\begin{equation}\label{PolicyIteration1}
     v_k \in \argmin_{v\in \mathcal M} \mathcal R_N^{\alpha_k}(v)
\end{equation}
using Algorithm \ref{algo:als}.

Update the policy according to
\begin{equation}\label{policy_update}
   \alpha_{k+1}(x) = -\frac{1}{2}B^{-1} g(x)^T \nabla v_k(x).
\end{equation}
Set $k = k+1$.
}
Set $\hat v = v_{k}$ and $\hat \alpha = - \frac 1 2 B^{-1} g^T \nabla \hat v $.
\end{algorithm}

\section{A preconditioner of the linearized Bellman equation}\label{sec:precon}
In this section we consider the question of how to choose the length of the trajectory $\tau$.

To this end, we consider the linear equation
\begin{equation}\label{eq:bellman}
 (\id - e^{-\gamma \tau} K_\tau^\alpha) v = \int_0^\tau e^{-\gamma t} r_t^\alpha dt.
\end{equation}
The first arising question is the existence of an inverse operator $\id - e^{-\gamma \tau}$, which is party covered in Remark \ref{rem:bounded_inverse}.
In particular we have $\| K_\tau^\alpha \|_{\mathcal L(L^\infty(\Omega))} = 1$, which means that for any discount factor $\gamma > 0$ the operator is a contraction and thus the inverse is given by the Neumann series
\begin{equation}
    (I - e^{-\gamma \tau} K_\tau^\alpha)^{-1} = \sum_{i = 0}^\infty (e^{-\gamma \tau}K_\tau^\alpha)^i.
\end{equation}
Observing that for small discount factor this operator becomes arbitrarily bad conditioned, a good candidate for a preconditioner of \eqref{eq:bellman} is $\sum_{i = 0}^N (e^{-\gamma \tau}K_\tau^\alpha)^i$.
However, due to the semigroup property of the Koopman operator, we have
\begin{align*}
        (\sum_{i=0}^N (e^{-\gamma\tau}K_{\tau}^\alpha)^i(\id - e^{-\gamma \tau} K_{\tau}^\alpha v))(x) = (\id - e^{-\gamma\tau(N+1)}K_{(N+1)\tau}^\alpha) v(x) = v(x) - e^{-\gamma\tau (N+1)} v(\Phi_{(N+1)\tau}(x)).
\end{align*}
In the same manner we obtain
\[\sum_{i=0}^N ((e^{-\gamma \tau}K_{\tau}^\alpha)^i \int_0^\tau e^{-\gamma t} r_t^\alpha(x) dt = \sum_{i=0}^N e^{-\gamma \tau i} \int_{i \tau  }^{(i+1)\tau} e^{-\gamma t} r_t^\alpha(x) dt. \]
Note that in particular we have 
\[\sum_{i=0}^N (e^{-\gamma \tau}K_{\tau}^\alpha)^i \int_0^\tau e^{-\gamma t} r_t^\alpha(x) dt = \int_0^{(N+1)\tau} e^{-\gamma t} r_t^\alpha(x) dt .\]

From this equation we see that calculating longer trajectories can be seen as a preconditioner of the Bellman equation.
\begin{remark}
Note that due to the Neumann series, we can solve \eqref{eq:bellman} pointwise by calculating trajectories with $\tau \to \infty$.
In this case, the linearized Bellman equation reduces to a regression problem, which is of course perfectly conditioned.
The numerics of this approach will, however, not be presented in this paper.
\end{remark}

\section{Numerical results}
We present results of numerical tests for different optimal control problems. For the implementation of the tensor networks we use the open source \code{c++} library \code{xerus} \citep{xerus}. 
The calculations were performed on a AMD Phenom II 6x 3.20GHz, 8 GB RAM openSUSE Leap 15.0 Linux distribution. In every test we consider a cost functional of the form
\begin{equation}\label{eq:cost}
    \argmin_{u \in L^2((0, \infty); \mathbb R^m)} \mathcal{J}(x, u) = \int_{0}^{\infty} \| y(t) \|^2 + 0.1 \|u(t)\|^2\ dt
\end{equation}
and a PDE, which we denote here as 
\[ \dot y = \tilde f(y) + g(y) u, \quad y \in \Theta\subset L^2(-1,1). \]
As the first step we discretize the PDE in space, such that we obtain a finite dimensional system of ODEs, which we also denote as
\[ \dot y = \tilde f(y) + g(y) u, \quad y \in \mathbb R^n. \]
Note that in order to fight the curse of dimensionality we apply tensor methods. Thus, we do not use an advanced method of discretization and instead use simple finite differences methods. We implement Algorithm \ref{algo:basic_pol_it} for the spatially discretized PDE. As polynomial ansatz spaces we use the tensor product of one-dimensional $H^1$-orthogonal polynomials as described in Section \ref{sect:TT}. The degree of the one-dimensional polynomials is specified in the test cases. \begin{remark}\label{rem:pictures}
In the following tests, we distinguish between the policy $\alpha$, the corrsponding cost estimator $v$ and the real generated cost $\mathcal J(\cdot, \alpha(\cdot))$. For fixed $x$, we obtain $v(x)$ by simply evaluating $v$. Here, no trajectory has to be computed. We obtain $\mathcal J(x, \alpha(x))$ by numerically integrating along the trajectory with initial condition $x$. Note that $\mathcal J(x, \alpha(x))$ is basically the numerical approximation of the cost functional with respect to a feedback law, defined in \eqref{eq:v_alpha}.
\end{remark}

\subsection{Test 1:  Diffusion with Unstable Reaction Term}
We consider a Schl\"ogl like system with homogeneous Neumann boundary condition, c.f. \cite[Test 2]{pol_approx_kunisch}. Solve \eqref{eq:cost} for $y \in \Theta = L^2(-1,1)$ subject to
\begin{align*}
\dot y &= \sigma \Delta y + y^3 + \chi_\omega u \\
y(0) &= x \\
\end{align*}
with homogeneous Neumann boundary condition and $\chi_\omega$ is the characteristic function w.r.t. $\omega = [-0.4, 0.4] \subset [-1,1]$. We choose $\sigma = 1$ and use a finite differences grid with $d \in \mathbb N$ grid points to discretize the spatial domain. We denote by $A$ this finite difference discritization of the Laplace operator and by $G\in \{0,1\}^d$ the discritization of the characteristic function $\chi_\omega$. Then we obtain a system of $d$ ordinary differential equations 

\begin{align*}
    &\dot y = \sigma A y + y^3 + Gu\\
    & y(0)=x
\end{align*}

Using the step-size $h = \frac{2}{d+1}$ we get a finite dimensional approximation of the term $\| y(t) \|_H^2$ in the cost functional. For this test we choose a spatial dimension of $n=32$. As the underlying equation is non linear, our ansatz for the value function is the tensor product of polynomials up to degree $4$. The internal ranks chosen are
\[ [3, 4, 5, 5, 5, 6, 6, 6, 6, 7, 7, 7, 7, 7, 7, 7, 7, 7, 7, 6, 6, 6, 6, 6, 6, 6, 5, 5, 5, 4, 3]. \]
We solve the HJB equation in 32 dimensions on the set $[-2,2]^d$. While the full ansatz space has dimension $5^{32}$, the TT has $5395$ degrees of freedom. We test two different loss functionals. First we choose the loss functional \eqref{eq:least_squares_penalty} with $\delta_1 = \delta_2 = 100$. We choose $\delta_3$ to be adaptive. In the beginning of one ALS sweep, we calculate the residuum $\tilde{\mathcal R}_N(v)$ and then set $\delta_3 = 10^{-3}\tilde{\mathcal R}_N(v)$ and keep it constant until the ALS sweep is complete. We denote the resulting value function by $v_{L_2}$ and the corresponding controller by $\alpha_{L^2}$. For the calculations we use $32768$ quasi Monte-Carlo samples.

\begin{remark}
We stress that the number of quasi Monte-Carlo samples is extremely small when comparing it to the dimension of the ambient space $32768 \ll 5^{32} \approx 10^{22}$ and only when comparing it to the degrees of freedom in the TT representation the numbers become comparable, $32768 \approx 6 \cdot 5395$. This further indicates, that the curse of dimensionality is broken by our ansatz. Additional studies, where the minimal number of samples is compared to the dimensions of the underlying systems have to be done.

\end{remark}

We introduce a second loss functional by incorporating information about $G^T \nabla v$ into the loss functional.  
The idea behind this approach is that for computing the control we do not only need a good approximation on $v$ but also of $\nabla v$. In particular, we need to evaluate the derivative of $v$ in direction $g = \chi_\omega$. Thus, by incorporating information about this derivative we hope to obtain an improved controller.

Denoting for a sample $x_i$ by $\tilde x_i = x_i + \varepsilon \chi_\omega$, we modify $\mathcal R_N(v)$, c.f. \eqref{eq:vmc_loss_functional}, by adding a discrete derivative
\ifnum\switch=2
\begin{equation}\label{eq:vmc_h1}
     \mathcal R_N^{H_1} (v) = \sum_{i=1}^N | \id - K_\tau^\alpha  v(x_i) - R(x_i)|^2 \\
     + | \frac {(\id - K_\tau^\alpha ) (v(\tilde x_i) - v(x_i)) - ( R(\tilde x_i) - R(x_i))} \varepsilon|^2
\end{equation} 
\fi
\ifnum\switch=1
\begin{multline}\label{eq:vmc_h1}
     \mathcal R_N^{H_1} (v) = \sum_{i=1}^N | \id - K_\tau^\alpha ) v(x_i) - R(x_i)|^2 \\
     + | \frac {(\id - K_\tau^\alpha ) (v(\tilde x_i) - v(x_i)) - ( R(\tilde x_i) - R(x_i))} \varepsilon|^2
\end{multline}
\fi
and plugging $\mathcal R_N^{H^1}$ into \eqref{eq:least_squares_penalty}.
For this cost functional we choose 16384 quasi Monte-Carlo samples and denote the value function by $v_{H_1}$ and the corresponding controller by $\alpha_{H^1}$. For every sample we calculate a trajectory with $1000$ steps of size $\tau = 0.001$ using the classical Runge-Kutta 4 method. We report that when we ran the same algorithms with only $100$ steps, the resulting controls where not better than the LQR controller. This shows the effect of the preconditioner, c.f. Section \ref{sec:precon}. In both cases we stop after 100 policy updates, where the relative difference between value functions was $~10^{-3}$. Note that so many iterations were necessary, because we stopped the ALS algorithm after 20 ALS sweeps, where the solution was not exact.

We also compute the optimal open loop control via a classical gradient descent method as described in \citep{troltzsch2010optimal}.
Here, we have to restrict ourselves to a finite time frame and choose $T = 5$ for computing the open loop control, because at this point the norm of the state was negligible.

We first test the feedback controllers for certain initial values, visualized in figure \ref{fig:schloegl_test_certain}. For both the $\alpha_{L_2}$ and the $\alpha_{H_1}$ controller, significant improvement in cost is noticable, with the greatest being approximately $50\%$ of the cost saved compared to the LQR controller. Moreover, we see that the computed feedback laws generate costs close to optimal costs for the tested initial values.
   \begin{figure}[htbp]
     \subfloat[Initial values.]{%
       \includegraphics[width=0.45\textwidth]{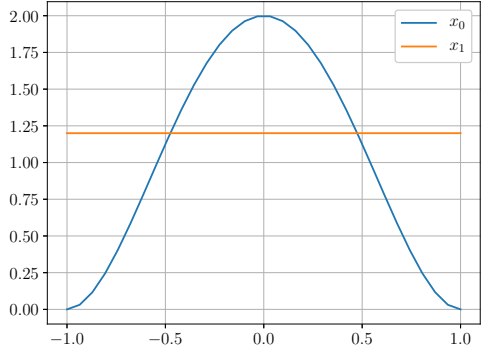}
     }
     \hfill
     \subfloat[Generated controls, initial value $x_{0}$.]{%
       \includegraphics[width=0.45\textwidth]{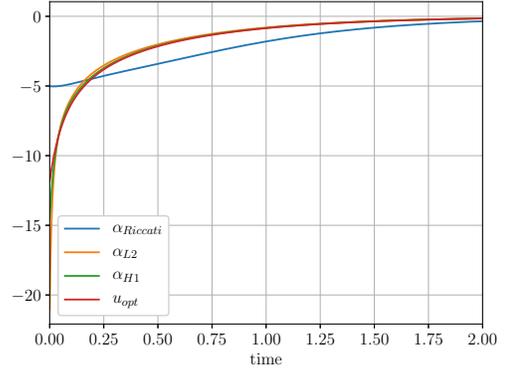}
     }
     \vfill
     \subfloat[Generated controls, initial value $x_{1}$.]{%
       \includegraphics[width=0.45\textwidth]{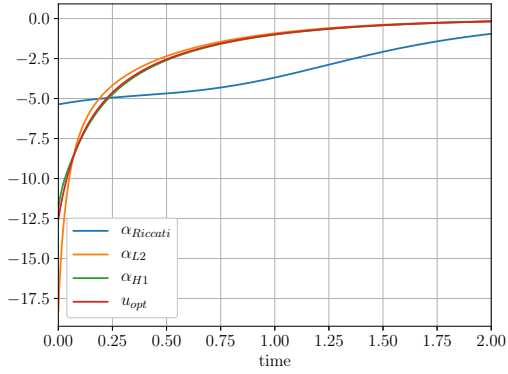}
     }
     \hfill
     \subfloat[Generated cost and Least-Squares error. Blue is Riccati, orange is $V_{L_2}$, black is $V_{H_1}$ and red is the open-loop optimal control. $v(x_i)$ is the approximated value function evaluated at $x_i$ and $\mathcal J(x_i, \alpha(x_i))$ is the actual cost, c.f. Remark \ref{rem:pictures}.\label{fig:S_f3}]{%
       \includegraphics[width=0.45\textwidth]{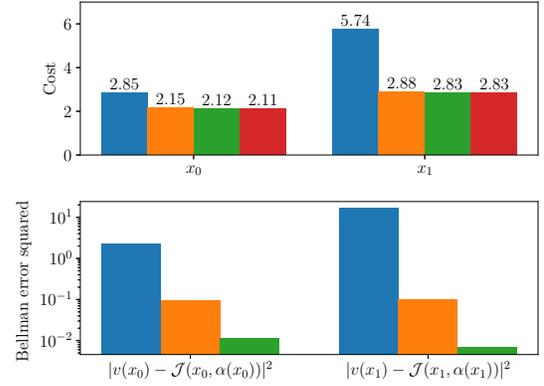}
     }
\caption{The generated controls for different initial values.}\label{fig:schloegl_test_certain}
   \end{figure}

Next we test the feedback law for random initial values. Note that because of the diffusion, equally distributed samples and normally distributed samples yield low cost on average and in this case no improvements of the cost is to be expected. Thus, we use a special distribution of initial values that we specify now. For every initial value we choose an equally distributed integer between $2$ and $20$. This number is the degree of a random polynomial. Next we choose a polynomial with normal distributed coefficients of the degree we chose. As this polynomial $\tilde p$ has its maximum in the interval $[-1, 1]$ on the boundary with high probability, we modify the polynomial in the following way $p(x) := \tilde p(x) (x-1)(x+1)$, such that we have $p(-1) = p(1) = 0$. 
Note that these initial values do not obey the Neumann boundary condition. However, due to our discretization this does not pose a problem.
Finally, we rescale $p$ such that its maximum in $[-1,1]$ is $1.75$. In order to have an idea how these initial values look, we plotted $10$ initial values in figure \ref{fig:B_random_initial}.
We report that for these initial values the LQR controller was not stabilizing in $443$ out of $1000$ initial values, while the $\alpha_{L_2}$ controller was not stabilizing in $12$ cases. The $\alpha_{H_1}$ controller was stabilizing for every initial value.

\begin{remark}
Note that from these initial values we can deduce that the area of attraction was increased by computing the $\alpha_{L_2}$ and $\alpha_{H_1}$ controller. This further underlines the importance of computing controllers for non linear systems. However, other methods for non linear systems, like Model Predictive Control, might also increase the area of attraction.
\end{remark}
We compare the average cost only for the initial values where LQR was stabilizing and we observe that even for these initial values, we obtain an average cost reduction of more than $25\%$. This is visualized in figure \ref{fig:S_random_costandbellman}.
   \begin{figure}[htbp]
     \subfloat[Percentage of unstable initial values for different controllers, 1000 initial values $\sim$ polynomial distribution.\label{fig:S_random_initial}]{%
       \includegraphics[width=0.45\textwidth]{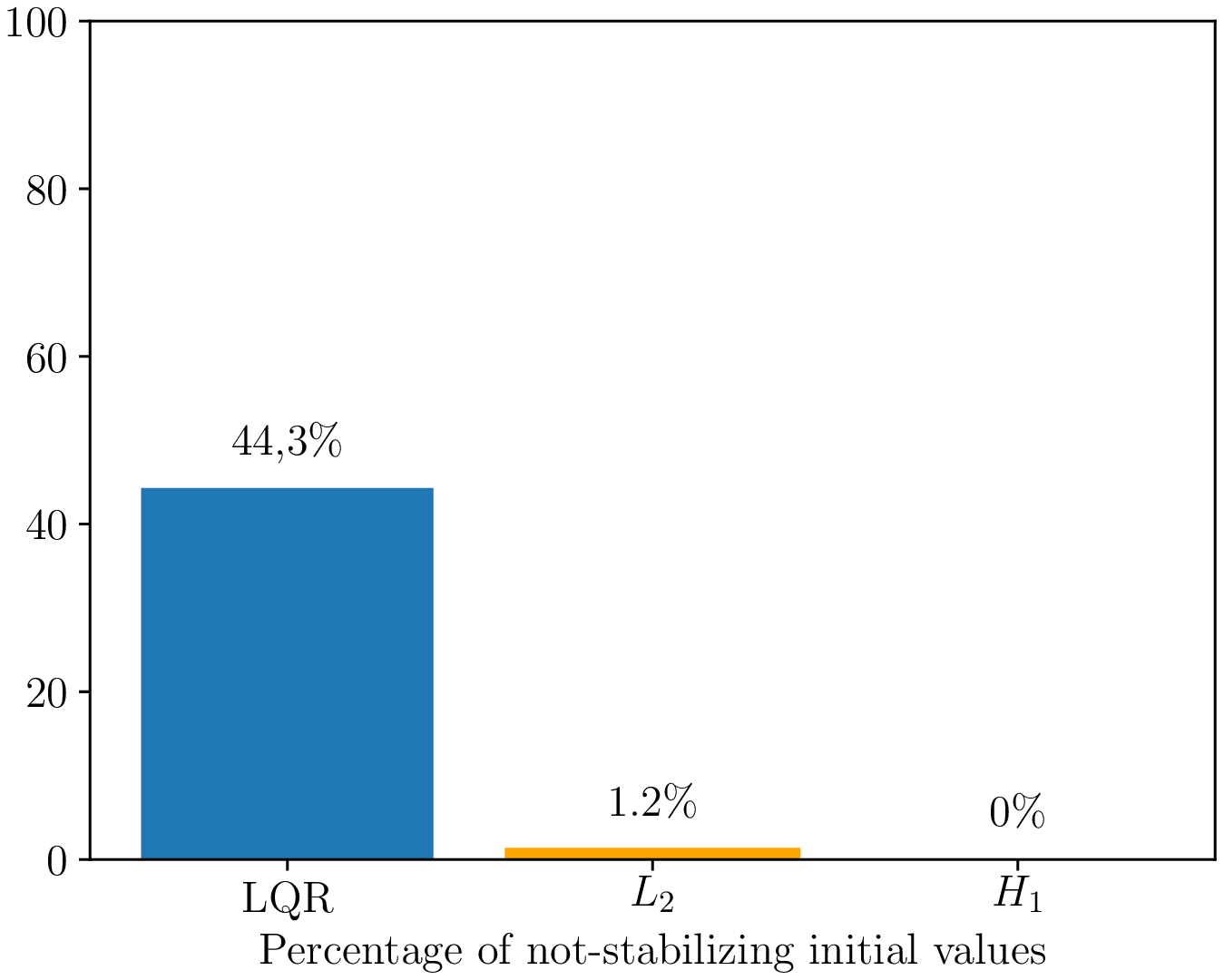}
     }
     \hfill
     \subfloat[Average cost for $1000$ initial values, left $x \sim \mathcal U(-3,3)$. Right: $x_1 \sim $ polynomial distribution, only the $557$ values where the LQR controller was stabilizing.]{%
       \includegraphics[width=0.45\textwidth]{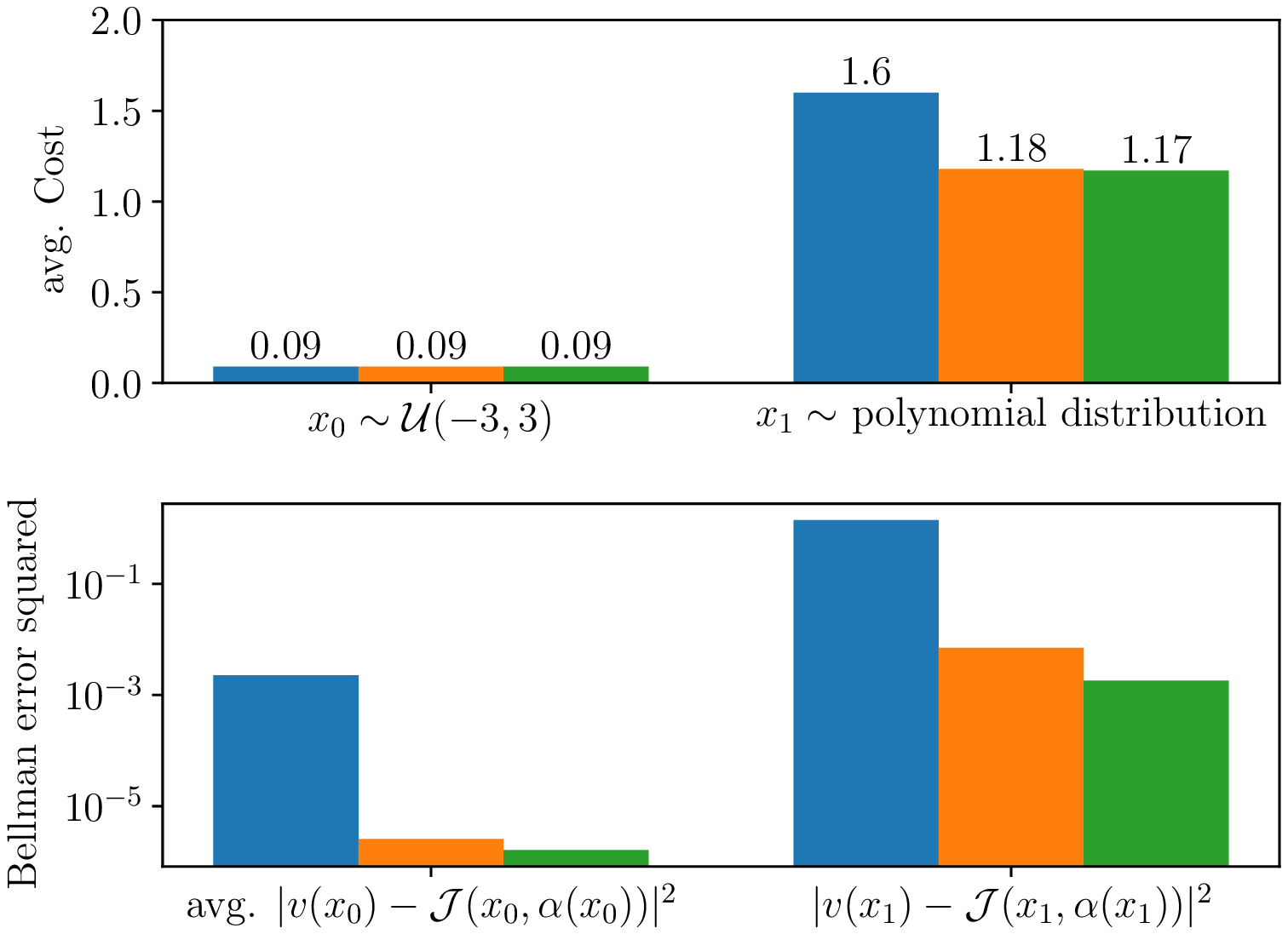}
     }
\caption{The generated cost for random initial values.}\label{fig:S_random_costandbellman}
   \end{figure}
\subsection{Test 2: Viscous Burgers'-like equation}\label{subsec:burgers}
As underlying equation we use a one-dimensional viscous Burgers'-like equation similar to \cite[Test 1]{pol_approx_kunisch}. Solve \eqref{eq:cost} for $y \in \Theta = L^2(-1,1)$ subject to
\begin{align*}
\dot y &= \sigma \Delta y +  \nabla (\frac{y^2}{2}) + 1.5 y e^{-0.1y} + \chi_\omega u \\
y(0) &= x
\end{align*}
with homogeneous Dirichlet boundary condition. Here, we use the same discretization as in Section \ref{subsec:burgers}. The constants are the same except for $\sigma = 0.2,\ \omega = [-0.5, 0.2]$. Again, an ansatz of polynomials up to degree $4$ is used. We choose the same ranks as in the last test
\[ [3, 4, 5, 5, 5, 6, 6, 6, 6, 7, 7, 7, 7, 7, 7, 7, 7, 7, 7, 6, 6, 6, 6, 6, 6, 6, 5, 5, 5, 4, 3] \]
and solve the HJB on $[-2,2]^{32}$. We stress that the choice of ranks is corresponds to the ranks of the LQR controller in TT-format, rounded with threshold $10^{-6}$. We calculate $v_{L_2}$ and $v_{H_1}$ as before only change the step size $\tau = 0.01$ and $\delta_3 = 10^{-6}\tilde{\mathcal R}_N(v)$.

We report that we also attempted computing an optimal open-loop control for this test, but did not succeed with both constant step size and with Armijos step size control. Instead, we added a control computed by the Nelder-Mead method provided by the python package SciPy \citep{SciPy}. We also report that while this method did not converge as well, improvement w.r.t. the cost was archieved and thus we included the result.

From figure \ref{fig:burgers_test_certain} we deduce that for certain initial values, significant improvement of cost is possible for both cost functional with the greatest improvement being $9.2\%$ of the cost. In both cases the $H^1$ cost functional gave small performance improvements, while the resulting cost is close to the approximated optimal cost. 
We again notice that for these initial values our calculated value functions are more accurate than the Riccati value function.

   \begin{figure}[htbp]
     \subfloat[Initial values.]{%
       \includegraphics[width=0.45\textwidth]{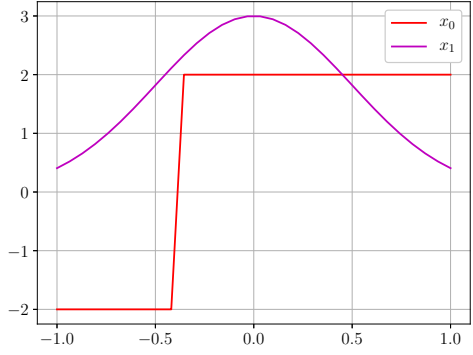}
     }
     \hfill
     \subfloat[Generated controls, initial value $x_{0}$]{%
       \includegraphics[width=0.45\textwidth]{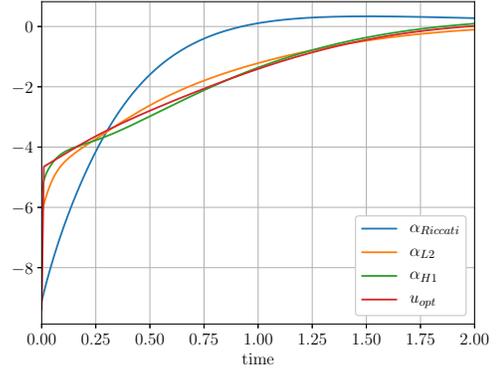}
     }
     \vfill
     \subfloat[Generated controls, initial value $x_{1}$.]{%
       \includegraphics[width=0.45\textwidth]{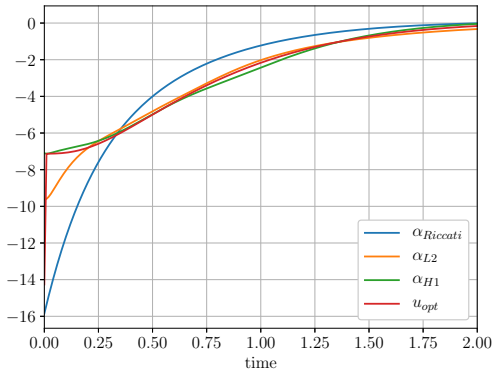}
     }
     \hfill
     \subfloat[Generated cost and Least-Squares error. Blue is Riccati, orange is $V_{L_2}$ and green is $V_{H_1}$. $v(x_i)$ is the approximated value function evaluated at $x_i$ and $\mathcal J(x_i, \alpha(x_i))$ is the actual cost, c.f. Remark \ref{rem:pictures}.\label{fig:B_f3}]{%
       \includegraphics[width=0.45\textwidth]{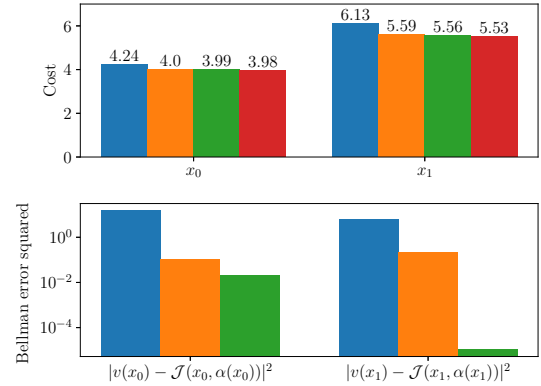}
     }
\caption{The generated controls for different initial values.}\label{fig:burgers_test_certain}
   \end{figure}

Next we again test random initial values. We are using the same setup as in test \ref{subsec:burgers}. The normalization of the polynomial random initial values is changed to $2.75$ instead of $1.75$ because the integration area was changed. In figure \ref{fig:B_random_costandbellman} we compare the performance of the controllers for $1000$ random initial values. For $x \sim \mathcal U(-3,3)$, no significant improvements in cost are visible, while the cost prediction for our value functions is more exact. The improvements are, however, not significant for these initial values. For $x$ distributed in the way described above there is a visible difference. On average, $6 \%$ of the cost is saved by $v_{L_2}$ and $v_{H_1}$. Moreover, the computed value functions predict their corresponding cost functional more exactly than the LQR value function.
   \begin{figure}[htbp]
     \subfloat[Examples of $10$ random initial values, drawn as described above, normalized such that the maximum is $2.75$.\label{fig:B_random_initial}]{%
       \includegraphics[width=0.45\textwidth]{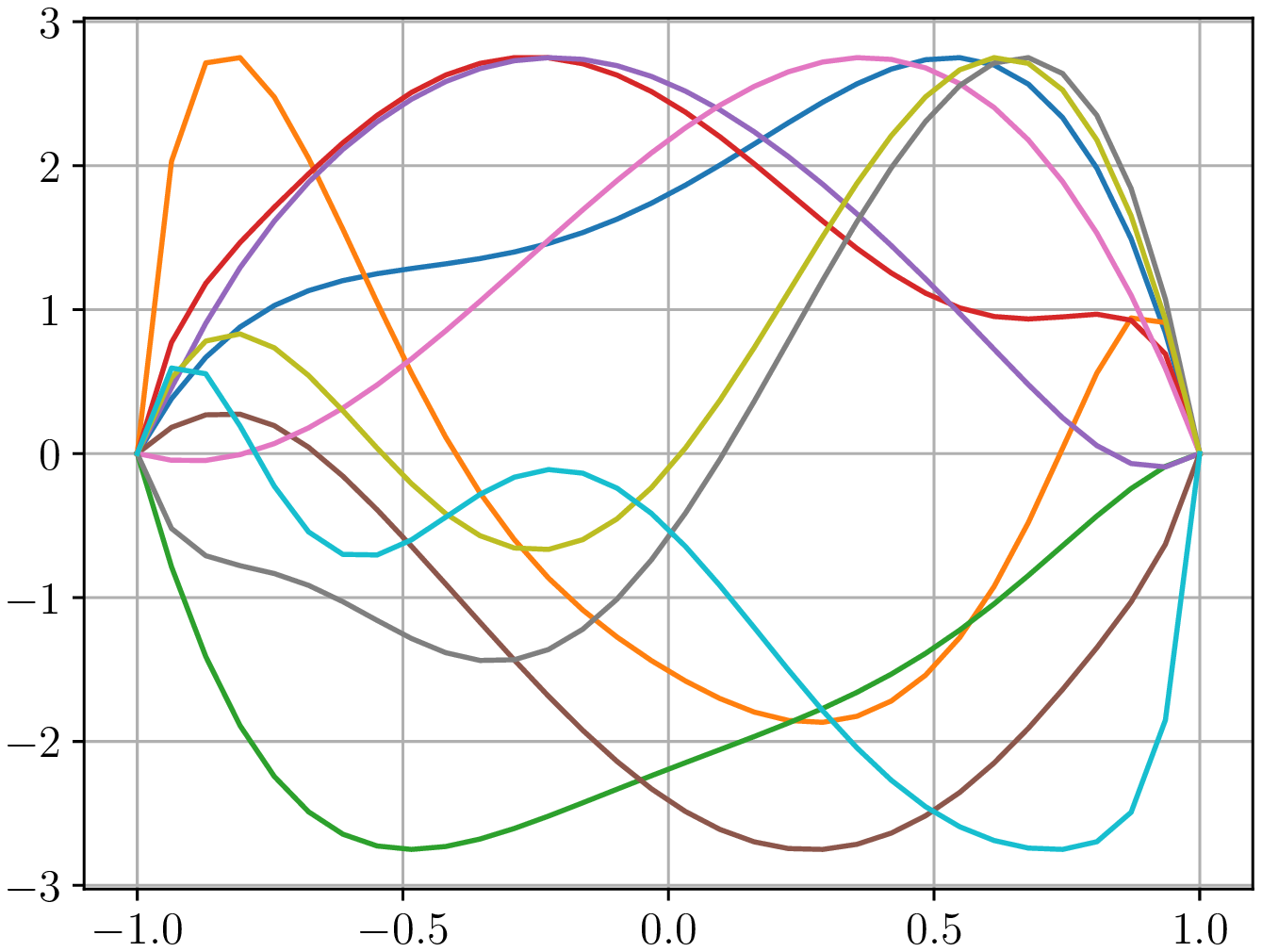}
     }
     \hfill
     \subfloat[Average cost for $1000$ initial values, left $x \sim \mathcal U(-3,3)$, right $x \sim $ polynomial distribution. $v(x_i)$ is the approximated value function evaluated at $x_i$ and $\mathcal J(x_i, \alpha(x_i))$ is the actual cost, c.f. Remark \ref{rem:pictures}.]{%
       \includegraphics[width=0.45\textwidth]{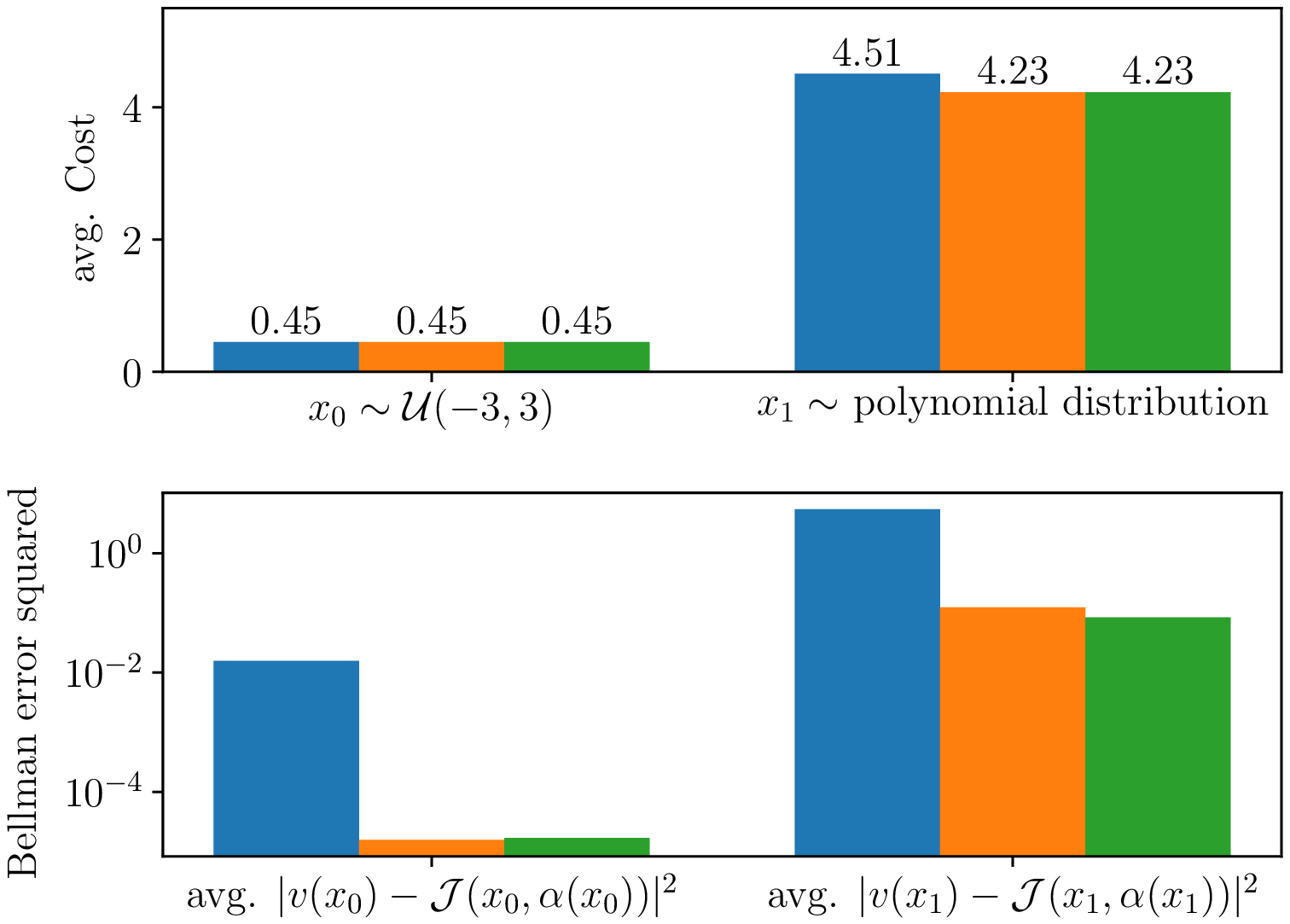}
     }
\caption{The generated cost for random initial values.}\label{fig:B_random_costandbellman}
   \end{figure}

\section{Conclusion}

We have considered a Policy Iteration in a minimal residual setting. 
Using variational Monte-Carlo, a straight forward implementation of the algorithm is possible. In our numerical studies we have demonstrated that for non linear tensor ansatz spaces the algorithm calculates near optimal controllers of simple structure, even in high dimensions. In real applications efficient evaluation of the controller is of utmost importance to enable online controlling.

Although we have consistently outperformed the LQR controller, we observe that the LQR controller works surprisingly well even in the non linear case. A careful and accurate treatment of the equations during the iteration procedure is mandatory to 
surpass this controller. Since we were able to find good approximations of the value function with polynomials of degree $4$, we conjecture that the value function can be approximated by smooth functions. In this direction there are recent theoretical results on the smoothness and Taylor expansion of certain PDE constrained optimal control problems, see \citep{BreitenStokes, BREITEN20191361}. 

Our present approach is closely related to dynamical programming and reinforcement learning, but we have made some essential modifications. Indeed, modern tools from machine learning like deep neural networks can be applied instead of hierarchical tensor representations.
For such smooth functions sparse polynomial 
approximations seems to be promising, too. 


\section{Acknowledgements}
Leon Sallandt, Reinhold Schneider and Mathias Oster acknowledge support from the Research Training Group ”Differential Equation- and Data-driven Models in Life Sciences and Fluid Dynamics: An Interdisciplinary Research Training Group (DAEDALUS)” (GRK 2433) funded by the German Research Foundation (DFG).
We would like to thank Fredi Tr\"oltzsch for sharing his broad insight to PDE constrained optimization. We also thank Michael G\"otte and Philipp Trunschke for fruitful discussions.
     \bibliography{HJB_and_HT}

\begin{thebibliography}{10}

\bibitem{FalconePOD}
A.~Alla, M.~Falcone, and S.~Volkwein.
\newblock Error analysis for {POD} approximations of infinite horizon problems
  via the dynamic programming approach.
\newblock {\em SIAM J. Control Optim.}, 55(5):3091--3115, 2017.

\bibitem{ALLA2020192}
Alessandro Alla and Luca Saluzzi.
\newblock A {HJB-POD} approach for the control of nonlinear {PDEs} on a tree
  structure.
\newblock {\em Applied Numerical Mathematics}, 155:192--207, 2020.
\newblock Structural Dynamical Systems: Computational Aspects held in Monopoli
  (Italy) on June 12-15, 2018.

\bibitem{azmi2020data}
Behzad Azmi, Dante Kalise, and Karl Kunisch.
\newblock Data-driven recovery of optimal feedback laws through optimality
  conditions and sparse polynomial regression.
\newblock {\em arXiv preprint arXiv:2007.09753}, 2020.

\bibitem{Dahmen-Bachmayr}
Markus Bachmayr, Albert Cohen, and Wolfgang Dahmen.
\newblock {Parametric PDEs: sparse or low-rank approximations?}
\newblock {\em IMA J. of Numerical Analysis}, 38(4):1661--1708, 09 2017.

\bibitem{Bachmayr-Uschmajew-Schneider}
Markus Bachmayr, Reinhold Schneider, and Andr{\'e} Uschmajew.
\newblock Tensor networks and hierarchical tensors for the solution of
  high-dimensional partial differential equations.
\newblock {\em Found. Comput. Math.}, 16(6):1423--1472, December 2016.

\bibitem{Barbu}
Viorel Barbu.
\newblock {\em Analysis and control of nonlinear infinite dimensional systems}.
\newblock Academic Press Boston, 1993.

\bibitem{DeterministicHJB}
Martino Bardi and Italo Capuzzo-Dolcetta.
\newblock {\em Optimal Control and Viscosity Solutions of
  {H}amilton-{J}acobi-{B}ellman Equations}.
\newblock Birk\"auser, 01 1997.

\bibitem{grohs-jensen}
Christian Beck, Sebastian Becker, Philipp Grohs, Nor Jaafari, and Arnulf
  Jentzen.
\newblock Solving stochastic differential equations and {K}olmogorov equations
  by means of deep learning, 2018.

\bibitem{Bertsekas}
Dimitri~P. Bertsekas.
\newblock {\em Dynamic Programming and Optimal Control}.
\newblock Athena Scientific, 2nd edition, 2000.

\bibitem{BreitenStokes}
Tobias Breiten, Karl Kunisch, and Laurent Pfeiffer.
\newblock Feedback stabilization of the two-dimensional {N}avier–{S}tokes
  equations by value function approximation.
\newblock {\em Applied Mathematics \& Optimization}, 06 2019.

\bibitem{BREITEN20191361}
Tobias Breiten, Karl Kunisch, and Laurent Pfeiffer.
\newblock Taylor expansions of the value function associated with a bilinear
  optimal control problem.
\newblock {\em Annales de l'Institut Henri Poincaré C, Analyse non linéaire},
  36(5):1361 -- 1399, 2019.

\bibitem{AppliedKoopmanism}
Marko Budišić, Ryan Mohr, and Igor Mezić.
\newblock Applied {K}oopmanism.
\newblock {\em Chaos: An Interdisciplinary J. of Nonlinear Science},
  22(4):047510, 2012.

\bibitem{schwabetal}
Abdellah Chkifa, Albert Cohen, and Christoph Schwab.
\newblock High-dimensional adaptive sparse polynomial interpolation and
  applications to parametric {PDEs}.
\newblock {\em Found. Comput. Math.}, 14(4):601--633, Aug 2014.

\bibitem{DaLio}
F.~Da~Lio.
\newblock On the {B}ellman equation for infinite horizon problems with
  unbounded cost functional.
\newblock {\em Applied Mathematics and Optimization}, 41(2):171--197, Apr 2000.

\bibitem{SemiLagrangianStochastic}
Kristian Debrabant and Espen Jakobsen.
\newblock Semi-{L}agrangian schemes for linear and fully non-linear
  {H}amilton-{J}acobi-{B}ellman equations.
\newblock In {\em Hyperbolic Problems: Theory, Numerics, Applications}, pages
  483--490. Springer, 03 2014.

\bibitem{TensorKunisch}
Sergey {Dolgov}, Dante {Kalise}, and Karl {Kunisch}.
\newblock {A Tensor Decomposition Approach for High-Dimensional
  {H}amilton-{J}acobi-{B}ellman Equations}.
\newblock {\em arXiv e-prints}, page arXiv:1908.01533, Aug 2019.

\bibitem{eigel2020convergence}
Martin Eigel, Reinhold Schneider, and Philipp Trunschke.
\newblock Convergence bounds for empirical nonlinear least-squares.
\newblock {\em arXiv preprint arXiv:2001.00639}, 2020.

\bibitem{VMC}
Martin Eigel, Reinhold Schneider, Philipp Trunschke, and Sebastian Wolf.
\newblock Variational {M}onte {C}arlo---bridging concepts of machine learning
  and high-dimensional partial differential equations.
\newblock {\em Advances in Computational Mathematics}, Oct 2019.

\bibitem{Falcone1987}
M.~Falcone.
\newblock A numerical approach to the infinite horizon problem of deterministic
  control theory.
\newblock {\em Applied Mathematics and Optimization}, 15(1):1--13, Jan 1987.

\bibitem{FALCONESplitting}
Maurizio Falcone, Piero Lanucara, and Alessandra Seghini.
\newblock A splitting algorithm for {H}amilton-{J}acobi-{B}ellman equations.
\newblock {\em Applied Numerical Mathematics}, 15(2):207 -- 218, 1994.

\bibitem{goodfellow}
Ian Goodfellow, Yoshua Bengio, and Aaron Courville.
\newblock {\em Deep Learning}.
\newblock The MIT Press, 2016.

\bibitem{Gota}
Maria Gota and Luigi Montrucchio.
\newblock On lipschitz continuity of policy functions in continuous-time
  optimal growth models.
\newblock {\em Economic Theory}, 14:479--488, 09 1999.

\bibitem{Hackbusch-buch}
Wolfgang Hackbusch.
\newblock {\em Tensor Spaces and Numerical Tensor Calculus}, volume~42.
\newblock Springer, 01 2012.

\bibitem{Hackbusch-Acta}
Wolfgang Hackbusch.
\newblock {Numerical tensor calculus}.
\newblock {\em Acta numerica}, 23:651--742, 2014.

\bibitem{Hackbusch2014}
Wolfgang Hackbusch and Reinhold Schneider.
\newblock {\em Tensor Spaces and Hierarchical Tensor Representations}.
\newblock Springer International Publishing, Cham, 2014.

\bibitem{ALS}
S.~Holtz, T.~Rohwedder, and R.~Schneider.
\newblock The alternating linear scheme for tensor optimization in the tensor
  train format.
\newblock {\em SIAM J. Sci. Comput.}, 34(2):A683--A713, 2012.

\bibitem{TTTensor}
Sebastian Holtz, Thorsten Rohwedder, and Reinhold Schneider.
\newblock On manifolds of tensors of fixed {TT}-rank.
\newblock {\em Numerische Mathematik}, 120(4):701--731, Apr 2012.

\bibitem{horowitz2014linear}
Matanya~B Horowitz, Anil Damle, and Joel~W Burdick.
\newblock Linear {Hamilton Jacobi Bellman} equations in high dimensions.
\newblock In {\em 53rd IEEE Conference on Decision and Control}, pages
  5880--5887. IEEE, 2014.

\bibitem{xerus}
Benjamin Huber and Sebastian Wolf.
\newblock Xerus - a general purpose tensor library.
\newblock {https://libxerus.org/}, 2014--2017.

\bibitem{VIM}
B.~Kafash, A.~Delavarkhalafi, and S.M. Karbassi.
\newblock Application of variational iteration method for
  {H}amilton-{J}acobi-{B}ellman.
\newblock {\em Applied Mathematical Modelling}, 37(6):3917 -- 3928, 2013.

\bibitem{pol_approx_kunisch}
Dante Kalise and Karl Kunisch.
\newblock Polynomial approximation of high-dimensional
  {H}amilton-{J}acobi-{B}ellman equations and applications to feedback control
  of semilinear parabolic {PDE}s.
\newblock {\em SIAM J. Sci. Comput.}, 40(2):A629--A652, 2018.

\bibitem{Khoromskij-book}
Boris~N. Khoromskij.
\newblock {Tensors-structured numerical methods in scientific computing :
  survey on recent advances}.
\newblock {\em Chemometrics and intelligent laboratory systems}, 110(1):1--19,
  2011.

\bibitem{kleinman}
D.~{Kleinman}.
\newblock On an iterative technique for {R}iccati equation computations.
\newblock {\em IEEE Transactions on Automatic Control}, 13(1):114--115,
  February 1968.

\bibitem{PerronFrobeniusKoopmanSchuette}
Stefan Klus, Peter Koltai, and Christof Sch\"utte.
\newblock On the numerical approximation of the {P}erron-{F}robenius and
  {K}oopman operator.
\newblock {\em J. of Computational Dynamics}, 3(2158-2491-2016-1-51):51, 2016.

\bibitem{Koopman315}
B.~O. Koopman.
\newblock Hamiltonian systems and transformation in {H}ilbert space.
\newblock {\em Proc. of the National Academy of Sciences}, 17(5):315--318,
  1931.

\bibitem{kundu}
Sudeep Kundu and Karl Kunisch.
\newblock Policy iteration for {Hamilton-Jacobi-Bellman} equations with control
  constraints.
\newblock 04 2020.

\bibitem{KUTSCHAN}
Benjamin Kutschan.
\newblock Tangent cones to tensor train varieties.
\newblock {\em Linear Algebra and its Applications}, 544:370 -- 390, 2018.

\bibitem{landsberg2012tensors}
Joseph~M Landsberg.
\newblock Tensors: geometry and applications.
\newblock {\em Representation theory}, 381(402):3, 2012.

\bibitem{LasotaMackey}
Andrzej Lasota.
\newblock {\em Chaos, fractals, and noise : stochastic aspects of dynamics}.
\newblock Applied mathematical sciences BV000005274 97. Springer, 2. ed.
  edition, 1994.

\bibitem{Beard}
J.~{Lawton} and R.~W. {Beard}.
\newblock Numerically efficient approximations to the
  {H}amilton-{J}acobi-{B}ellman equation.
\newblock In {\em Proceedings of the 1998 American Control Conference. ACC
  (IEEE Cat. No.98CH36207)}, volume~1, pages 195--199 vol.1, June 1998.

\bibitem{LiYongOptContrInfDim}
Xunjing Li and Jiongmin Yong.
\newblock {\em Optimal Control Theory for Infinite Dimensional Systems}.
\newblock Birkh\"auser, 1995.

\bibitem{DataHJB}
Biao Luo, Huai-Ning Wu, Tingwen Huang, and Derong Liu.
\newblock Data-based approximate policy iteration for affine nonlinear
  continuous-time optimal control design.
\newblock {\em Automatica}, 50(12):3281 -- 3290, 2014.

\bibitem{Oseledets2}
Ivan Oseledets.
\newblock Tensor-train decomposition.
\newblock {\em SIAM J. Sci. Comput.}, 33:2295--2317, 01 2011.

\bibitem{Oseledets}
Ivan Oseledets and E.~Tyrtyshnikov.
\newblock Breaking the curse of dimensionality, or how to use {SVD} in many
  dimensions.
\newblock {\em SIAM J. Sci. Comput.}, 31:3744--3759, 01 2009.

\bibitem{Puterman}
Martin~L. Puterman and Shelby~L. Brumelle.
\newblock On the convergence of {P}olicy {I}teration in stationary dynamic
  programming.
\newblock {\em Mathematics of Operations Research}, 4(1):60--69, 1979.

\bibitem{RAISSI2019686}
M.~Raissi, P.~Perdikaris, and G.E. Karniadakis.
\newblock Physics-informed neural networks: A deep learning framework for
  solving forward and inverse problems involving nonlinear partial differential
  equations.
\newblock {\em Journal of Computational Physics}, 378:686--707, 2019.

\bibitem{RAUHUT2017220}
Holger Rauhut, Reinhold Schneider, and Željka Stojanac.
\newblock Low rank tensor recovery via iterative hard thresholding.
\newblock {\em Linear Algebra and its Applications}, 523:220 -- 262, 2017.

\bibitem{RINCONZAPATERO2012305}
Juan~Pablo Rincón-Zapatero and Manuel~S. Santos.
\newblock Differentiability of the value function in continuous-time economic
  models.
\newblock {\em Journal of Mathematical Analysis and Applications}, 394(1):305
  -- 323, 2012.

\bibitem{Saridis}
G.~N. {Saridis} and C.~S. {G. Lee}.
\newblock Optimal control approximations for trainable manipulators.
\newblock In {\em 1977 IEEE Conference on Decision and Control including the
  16th Symposium on Adaptive Processes and A Special Symposium on Fuzzy Set
  Theory and Applications}, pages 749--754, Dec 1977.

\bibitem{schoelkopf}
Bernhard Scholkopf and Alexander~J. Smola.
\newblock {\em Learning with Kernels: {S}upport {V}ector {M}achines,
  Regularization, Optimization, and Beyond}.
\newblock MIT Press, Cambridge, MA, USA, 2001.

\bibitem{sickel2009tensor}
Winfried Sickel and Tino Ullrich.
\newblock Tensor products of {S}obolev-{B}esov spaces and applications to
  approximation from the hyperbolic cross.
\newblock {\em Journal of Approximation Theory}, 161(2):748--786, 2009.

\bibitem{CompositionOpLp}
R.K. Singh and J.S. Manhas.
\newblock Chapter ii composition operators on lp-spaces.
\newblock In {\em Composition Operators on Function Spaces}, volume 179 of {\em
  North-Holland Mathematics Studies}, pages 17 -- 58. North-Holland, 1993.

\bibitem{steinwart}
Ingo Steinwart and Andreas Christmann.
\newblock {\em {S}upport {V}ector {M}achines}.
\newblock Springer Publishing Company, Incorporated, 1st edition, 2008.

\bibitem{Legeza-Schneider}
Szilárd Szalay, Max Pfeffer, Valentin Murg, Gergely Barcza, Frank Verstraete,
  Reinhold Schneider, and Örs Legeza.
\newblock {Tensor product methods and entanglement optimization for ab initio
  quantum chemistry}.
\newblock {\em International j. of quantum chemistry}, 115(19):1342--1391,
  2015.

\bibitem{SemiLagranigian}
Daniela Tonon, Maria Aronna, and Dante Kalise.
\newblock {\em Optimal Control: Novel Directions and Applications}.
\newblock Springer, 01 2017.

\bibitem{troltzsch2010optimal}
F.~Tr{\"o}ltzsch and J.~Sprekels.
\newblock {\em Optimal Control of Partial Differential Equations: Theory,
  Methods, and Applications}.
\newblock Graduate studies in mathematics. American Mathematical Society, 2010.

\bibitem{schaft}
Arjan {van der Schaft}.
\newblock ${L}_2$-gain analysis of nonlinear systems and nonlinear state
  feedback ${H}_\infty$ control.
\newblock {\em IEEE transactions on automatic control}, 37(6):770--784, 1992.

\bibitem{KoopmanofRandomDynSys}
Nelida {{\v{C}}rnjari{\'c}-{\v{Z}}ic}, Senka {Ma{\'c}e{\v{s}}i{\'c}}, and Igor
  {Mezi{\'c}}.
\newblock {{K}oopman Operator Spectrum for Random Dynamical Systems}.
\newblock {\em arXiv e-prints}, page arXiv:1711.03146, Nov 2017.

\bibitem{SciPy}
Pauli {Virtanen}, Ralf {Gommers}, Travis~E. {Oliphant}, Matt {Haberland}, Tyler
  {Reddy}, David {Cournapeau}, Evgeni {Burovski}, Pearu {Peterson}, Warren
  {Weckesser}, Jonathan {Bright}, St{\'e}fan~J. {van der Walt}, Matthew
  {Brett}, Joshua {Wilson}, K.~{Jarrod Millman}, Nikolay {Mayorov}, Andrew
  R.~J. {Nelson}, Eric {Jones}, Robert {Kern}, Eric {Larson}, CJ~{Carey},
  {I}lhan {Polat}, Yu~{Feng}, Eric~W. {Moore}, Jake {Vand erPlas}, Denis
  {Laxalde}, Josef {Perktold}, Robert {Cimrman}, Ian {Henriksen}, E.~A.
  {Quintero}, Charles~R {Harris}, Anne~M. {Archibald}, Ant{\^o}nio~H.
  {Ribeiro}, Fabian {Pedregosa}, Paul {van Mulbregt}, and SciPy 1.~0
  {Contributors}.
\newblock {SciPy 1.0: Fundamental Algorithms for Scientific Computing in
  Python}.
\newblock {\em Nature Methods}, 17:261--272, 2020.

\end{thebibliography}

\end{document}